\documentclass[11pt, letterpaper]{amsart}

\usepackage[margin=1.2in]{geometry}

\makeatletter
\def\@textbottom{\vskip \z@ \@plus 7pt}
\let\@texttop\relax
\makeatother

\usepackage{tikz-cd}
\tikzcdset{arrow style=tikz, diagrams={> ={Straight Barb[scale=0.85]}}}
\usepackage{txfonts}
\usepackage{mathtools}

\usepackage{booktabs, array, multirow, longtable}
\usepackage{rotating}
\usepackage{microtype}
\usepackage[T1]{fontenc}
\usepackage[msc-links]{amsrefs}

\usepackage{xurl}
\usepackage{hyperref}
\usepackage[all]{hypcap}

\usepackage{amsthm}
\usepackage{cleveref}

\theoremstyle{plain}
\newtheorem{theo}{Theorem}[section]
\newtheorem{prop}[theo]{Proposition}
\newtheorem{lemm}[theo]{Lemma}
\newtheorem{coro}[theo]{Corollary}

\theoremstyle{definition}
\newtheorem{defn}{Definition}[section]

\theoremstyle{remark}
\newtheorem*{rema}{Remark}
\crefdefaultlabelformat{#2\textup{#1}#3}

\crefname{equation}{equation}{equations}
\crefname{table}{Table}{Tables}
\crefformat{enumi}{item~#2(#1)#3}
\Crefformat{enumi}{Item~#2(#1)#3}
\crefrangeformat{enumi}{items~#3\textup{(#1)}#4~to~#5\textup{(#2)}#6}
\Crefrangeformat{enumi}{Items~#3\textup{(#1)}#4~to~#5\textup{(#2)}#6}
\crefname{theo}{Theorem}{Theorems}
\crefname{prop}{Proposition}{Propositions}
\crefname{defn}{Definition}{Definitions}
\crefname{coro}{Corollary}{Corollaries}
\crefname{lemm}{Lemma}{Lemmata}

\newcommand{\Sym}{\mathfrak{S}}
\newcommand{\Alt}{\mathfrak{A}}
\newcommand{\CC}{\mathbb{C}}
\newcommand{\PP}{\mathbb{P}}
\newcommand{\QQ}{\mathbb{Q}}
\newcommand{\ZZ}{\mathbb{Z}}
\newcommand{\Id}{\mathrm{Id}}
\newcommand{\Cyc}{\mathrm{C}}
\newcommand{\Dih}{\mathrm{D}}

\newcommand{\GA}{\operatorname{Aff}(\mathbb{F}_5)}

\DeclareMathOperator{\Mon}{Mon}
\DeclareMathOperator{\Aut}{Aut}
\DeclareMathOperator{\Irr}{Irr}
\DeclareMathOperator{\Ind}{Ind}
\DeclareMathOperator{\End}{End}
\DeclareMathOperator{\Altp}{Alt}
\DeclareMathOperator{\ima}{im}

\DeclareMathOperator{\N}{N}
\DeclareMathOperator{\J}{J}
\DeclareMathOperator{\mult}{mult}
\DeclareMathOperator{\lcm}{lcm}
\DeclareMathOperator{\type}{type}
\DeclareMathOperator{\Stab}{Stab}
\DeclareMathOperator{\Prym}{P}
\DeclareMathOperator{\fg}{\pi_1}

\DeclarePairedDelimiter{\paren}{\lparen}{\rparen}
\DeclarePairedDelimiter{\ord}{\lvert}{\rvert}
\DeclarePairedDelimiter{\angp}{\langle}{\rangle}

\title[Galois closure of a 5-fold covering]{Galois closure of a 5-fold covering\\ and decomposition of its Jacobian}
\author{Benjamín M.\@ Moraga Baeza}
\email{\href{mailto:benjamin.baeza@ufrontera.cl}{benjamin.baeza@ufrontera.cl}}
\address{Programa de Doctorado en Ciencias menci\'on Matem\'atica, Universidad de La Frontera, Temuco, Chile}
\thanks{Partially supported by ANID Fondecyt grants 1190991.}

\begin{document}
	\begin{abstract}
		For an arbitrary 5-fold ramified covering between compact Riemann surfaces, every possible Galois closure is determined in terms of the ramification data of the map; namely, the ramification divisor of the covering map. Since the group that acts on the Galois closure also acts on the Jacobian variety of the covering surface, we describe its group algebra decomposition in terms of the Jacobian and Prym varieties of the intermediate coverings of the Galois closure. The dimension and induced polarization of each abelian variety in the decomposition is computed in terms of the ramification data of the covering map.
	\end{abstract}

	\maketitle


	\section{Introduction}
    
    In the moduli space \(\mathcal A_g\) of isomorphism classes of principally polarized abelian varieties of dimension \(g\), the subspace of isomorphism classes of Jacobian varieties deserve special attention. Due to Torelli's theorem, to each curve \(X\) of genus \(g\) we can biunivocally associate a principally polarized abelian variety \(\J(X)\) of dimension \(g\); in this manner, we can think the moduli space \(\mathcal{M}_g\) of the Riemann surfaces of genus \(g\) as a special subvariety of \(\mathcal A_g\), the locus of the Jacobian varieties. This particular kind of principally polarized abelian varieties were the first ones to be studied and they are still the best known among the latter ones due to the information we can naturally get from their associated curve.
        
    Poincare's complete reducibility theorem (see \cite{book:Lange}) states that each abelian variety is isogenous to a product of simple abelian subvarieties, those subvarieties are uniquely determined up to isogeny; however there is no easy manner to make that decomposition explicit. When we have a group action on an abelian variety, we have its isotypical decomposition, which is usually coarser than the given by Poincare's theorem, into \(G\)-invariant abelian sub-varieties and the, slightly finer, group algebra decomposition. Moreover, in a Jacobian variety \(\J(X)\), when that action is inherited by an action on \(X\), the pieces of the group algebra decomposition may, in many cases, be described as Prym varieties of intermediate coverings of the Galois covering induced by the group action on \(X\). The first example of this phenomenon is the Recillas trigonal construction (see \cite{paper:recillasTri}) that shows that the Jacobian of a tetragonal curve is isomorphic to the Prym variety of a double cover of a trigonal curve. Later, in \cite{paper:RR98} and \cite{paper:RR03}, Recillas and Rodriguez generalized those results for curve coverings of degree 3 and 4. They also analyze the polarization types of the isogenies involved. In this article we study analogous phenomena for degree~5 coverings in the light of the more general results of \cite{paper:CR2006}. Actions of \(\Alt_5\) and \(\Sym_5\) on Jacobian varieties are also studied on \cite{paper:Sanchez,paper:LangeRecillas04}; those are particular cases of the Galois closure of a degree~5 covering.

	Throughout this article, let \(f\colon X\to Y\) denote a ramified covering of degree~\(5\) between compact Riemann surfaces with branch locus \(B\) and monodromy representation \(\rho\colon \fg(Y-B) \to \Sym_5\). Although \(f\) is not necessarily Galois (that is, the quotient of \(X\) by a group action), its Galois closure \(\hat f\colon \hat X\to Y\) is. This action induces another one on the Jacobian variety \(\J(\hat{X})\); hence, we can decompose \(\J(\hat{X})\) into smaller abelian varieties through the group algebra decomposition (see \cite{paper:CR2006}). Since the automorphism group \(\Aut(\hat f)\) is naturally isomorphic to the monodromy group \(\Mon(f)\), the geometric signature (see \cite{thesis:AnitaRojas}) of \(\hat f\) is determined by the ramification data of \(f\); hence, the group algebra decomposition of \(\J(\hat X)\) depends on the ramification data of \(f\).

    There are two main results in this article: in \cref{theo:ClassificationG0,theo:ClassificationGP}, we enumerate each possible monodromy group \(\Mon(f)\), up to conjugacy in \(\Sym_5\), in terms of the ramification data of \(f\); we also give explicit generating vectors for each possible case. Then, and as an application, in \cref{theo:CaseC5,theo:CaseD5,theo:CaseGA,theo:CaseA5,theo:CaseS5}, we give the group algebra decomposition of \(\J(\hat X)\) for each possible \(\Mon(f)\). 

	\section{Preliminaries} \label{sec:Preliminaries}

	\subsection{Galois closure and its intermediate coverings} \label{subsec:GaloisClosure}

	The Galois closure \(\hat f\colon \hat X \to Y\) is the minimal Galois covering that factors through \(f\); the action of \(\fg\paren{Y-B}\) on the universal covering of \(Y-B\) induces a natural isomorphism \(\Aut(\hat f) \cong \Mon(f)\), so we refer to elements of \(\Aut(\hat f)\) and \(\Mon(f)\) interchangeably (see \cite{book:Girondo}*{section~2.8}); also, for simplicity, we set \(G \coloneqq \Mon(f)\) and denote \(g \coloneqq g_Y\) throughout this article. For each subgroup \(H\) of \(G\), there is a quotient map \(\pi_H\colon \hat X \to \hat X/H\) and an induced holomorphic map \(\pi^H\colon \hat X/H \to Y\), called \emph{intermediate coverings} of \(\hat f\), such that \(\hat f = \pi^H \circ \pi_H\).
	Moreover, we have \(\Aut(\pi^H) \cong \N_G(H)/H\) and \(X \cong \hat X/\Stab_G(1)\). Conversely, each intermediate covering of \(\hat f\) is induced by the action of a subgroup of \(G\). The following result summarize some direct consequences of that correspondence.

	\begin{prop}\label{prop:IntermediateOfGalois}
\hfill
		\begin{enumerate}
			\item The map \(\pi^H\) has a non-trivial intermediate covering if and only if there is a subgroup \(K\) of \(G\) such that \(H \subsetneq K \subsetneq G\).

			\item The map \(\pi^H\) is Galois if and only if \(H\) is normal in \(G\). If \(\pi^H\) is Galois, then \(\Aut(\pi^H) \cong G/H\).
		\end{enumerate}
	\end{prop}

	For subgroups \(H\) and \(N\) of \(G\) such that \(H\subseteq N\), we define the induced map \(\pi_N^H\colon \hat{X}/H \to \hat{X}/N\) as the intermediate covering of \(\pi_N\) induced by the action of \(H\).

	\begin{defn}\label{defn:Type}
		For a branch value \(y\in B\) with \(f^{-1}(y) = \{x_1,\ldots, x_k\}\), we define the \emph{type} of \(y\) as \([\mult_{x_1}(f), \ldots, \mult_{x_k}(f)]\).
	\end{defn}
	The type of a branch value equals the cycle structure of \(\rho(\tau_y)\), where \(\tau_y\) is a small loop around \(y\); so we say that \(y\) is \emph{even} or \emph{odd} according to the parity of \(\rho(\tau_y)\).

	\begin{defn}
		The \emph{ramification data} of \(f\) is its set of branch values and their respective types, say \(t_i\). We denote it by \((t_1,\ldots,t_n)\).
	\end{defn}

	Denote \(B = \{y_1,\ldots,y_n\}\), the ramification data \((t_1,\ldots,t_n)\) of \(f\) is related with the signature \((g; m_1,\ldots, m_n)\) and the geometric signature  \((g;C_1,\ldots,C_n)\) of \(\hat f\) as follows (see \cite{thesis:AnitaRojas}):
	\begin{enumerate}
		\item Each \(C_i\) is the \(G\)-conjugacy class of the subgroup generated by a permutation of cycle structure \(t_i\); namely \(\rho(\tau_y)\).

		\item If \(t_i = [\nu_{i,1},\ldots,\nu_{i,k_i}]\), then \(m_i = \lcm(\nu_{i,1},\ldots,\nu_{i,k_i})\).
	\end{enumerate}
    Hence, the following result follows from the Riemann's existence theorem (see \cite{paper:Broughton}).

	\begin{prop}\label{prop:ExistenceOfHolomorphicMaps}
        The following statements are equivalent:
		\begin{enumerate}
			\item\label{stat:ETHolomorphicMap} There is a \(d\)-fold covering map \(f\colon X\to Y\) with branch locus \(\{y_1, \ldots, y_n\}\) such that \(\Mon(f) = G\) and \(\type(y_i) = [\nu_{i, 1}, \ldots, \nu_{i, {k_i}}]\) for each \(i\in\{1,\ldots,n\}\).

			\item\label{stat:ETGeneratingVector} The group \(G\) has a generating \((g; m_1, \ldots, m_n)\)-vector \((a_1, \ldots, a_{g}, b_1,\ldots,\allowbreak b_{g}, c_1, \ldots, c_n)\) such that each \(c_i\) has cycle structure \([\nu_{i, 1}, \ldots, \nu_{i, {k_i}}]\).
		\end{enumerate}
	\end{prop}

	\begin{rema}
		Recall that a generating \((g; m_1, \ldots, m_n)\)-vector of a group \(G\) is a tuple \((a_1, b_1, \ldots, a_{g}, b_{g}, \allowbreak c_1, \ldots, c_n)\) of generators of \(G\) such that
		\begin{equation}\label{eq:genVector}
			\prod_{i=1}^{g}[a_i,b_i] \prod_{i=1}^{n} c_i = \Id.
		\end{equation}
	\end{rema}



    The genus \(g_{\hat X/H}\) and the ramification data of \(\pi^H\) and \(\pi_H\) can be computed from the geometric signature of \(\hat f\) as done in \cite{thesis:AnitaRojas}*{chapter~3}; these computations were implemented into a Sage class (see \cite{software:SageMath}) called \texttt{GaloisCovering} using mostly GAP functions (see \cite{software:GAP}); its source code can be found at \url{https://bit.ly/3wDmLiL}. There is also a previous (although non object-oriented) implementation through GAP, which can be found in \cite{thesis:AnitaRojas}*{appendix~A}.

	\subsection{Group representations}

	For a finite group \(G\) and a complex representation \(V\) of \(G\), we denote its character by \(\chi_V\) and its Schur index by \(m_V\).
    For simplicity, we refer to a rational representation \(W\) and the complex representation \(W\otimes \CC\) interchangeably.

	\subsubsection{Group algebra decomposition}

	If \(\Irr_\QQ(G) = \{W_1, \ldots, W_r\}\), there is a unique decomposition
	\begin{equation}\label{eq:QGdecomposition}
		\QQ[G] = \QQ[G]e_1 \oplus \cdots \oplus \QQ[G]e_r
	\end{equation}
	of \(\QQ[G]\) into simple \(\QQ[G]\)-algebras, where the \(e_i\) are mutually orthogonal central idempotents associated to the irreducible rational representations \(W_i\) of \(G\). \Cref{eq:QGdecomposition} is called the \emph{isotypical decomposition} of \(\QQ[G]\). Furthermore, for each \(e_j\) there are primitive orthogonal idempotents \(q_{j,1}, \ldots, q_{j, l_j} \in \QQ[G] e_j\) such that \(\QQ[G] e_j = \QQ[G] q_{j, 1} \oplus \cdots \oplus \QQ[G] q_{j, l_j}\)
	is a decomposition of \(\QQ[G] e_j\) into minimal right ideals. This decomposition is only determined up to isomorphism. Moreover, if \(V_j \in \Irr_\CC(G)\) is Galois associated to \(W_j\), then \(l_j = \dim V_j/m_{V_j}\).

	\begin{defn}\label{defn:GroupAlgebraDecomposition}
		The decomposition
		\[
			\QQ[G] = \bigoplus_{j = 1}^{r} \bigoplus_{i = 1}^{l_j} \QQ[G]q_{j, i},
		\]
		unique up to isomorphism, is called the \emph{group algebra decomposition}.
	\end{defn}

	\subsubsection{Representations induced by a trivial representation}

	For a subgroup \(H\) of \(G\), we set \(\rho_H  \coloneqq \Ind_H^G (1_H)\), where \(1_H\) denotes the trivial representation of \(H\) (see \cite{book:Serre}); since \(1_H\) is rational, so it is \(\rho_H\). Also, we denote the character of \(\rho_H\) by \(\chi_H\).
    We can decompose \(\rho_H\) through Frobenious reciprocity and the rational character tables of \(H\) and \(G\).

	\subsection{Group algebra decomposition by Prym varieties} For a curve \(Z\), we denote its Jacobian variety by \((\J(Z), \Theta_Z)\).

	\subsubsection{Prym varieties}

	The pullback \(f^* \colon \J(Y) \to \J(X)\) is a homomorphism with finite kernel; thus an \emph{isogeny} onto \(f^*\J(Y)\). Also, we have the following result.

	\begin{prop}[\cite{book:Lange}*{Proposition~11.4.3}]\label{prop:CyclicEtaleFactor}
		The homomorphism \(f^*\) is not injective if and only if \(f\) factorizes via a cyclic unramified covering \(f'\) of degree greater than \(2\) as in the following commutative diagram:
		\[
		\begin{tikzcd}
			X \ar[rr, "f"] \ar[dr] & & Y \\
			& Z \ar[ur,"f'"']
		\end{tikzcd}
		\]
		In this case, we have that \(\ker f'^*\) is cyclic of order \(\deg f'\).
	\end{prop}

	Each abelian subvariety has a unique complementary abelian subvariety (see \cite{book:Lange}*{p.~125}), so the following definition make sense.

	\begin{defn}
		The \emph{Prym variety} of \(f\colon X \to Y\), denoted by \(\Prym(f)\), is the complement of \(f^*\J(Y)\) in \(\J(X)\) with respect to \(\Theta_X\).
	\end{defn}

	With respect to the polarization induced on \(\Prym(f)\), we have the following result.

	\begin{theo}[\cite{paper:RR03}*{Theorem~2.5}]\label{theo:PrymAndPolarizations}
		Let \(f\colon X\to Y\) be a covering map of degree \(d\). Denote by \(\Theta_{f^*\J(Y)}\) and by \(\Theta_{\Prym(f)}\) the polarizations induced by \(\Theta_X\) in \(f^*\J(Y)\) and \(\Prym(f)\), respectively. Then:
		\begin{enumerate}
			\item\label{item:InducedPolarizationOnJacobian} The pullback of \(\Theta_{f^*\J(Y)}\) by \(f^*\) is analytically equivalent to \(\Theta_Y^{\otimes d}\).

			\item\label{item:DecompositionInJacAndPrym} The homomorphism \(f^*\) induces an isomorphism
			\[
				f^* \colon \dfrac{(\ker f^*)^{\perp}}{\ker f^*} \to \ker (\Theta_{f^*\J(Y)})\text{.}
			\]
			Moreover, we have \(\ker (\Theta_{\Prym(f)}) = \ker (\Theta_{f^*\J(Y)})\).
		\end{enumerate}
	\end{theo}

	Now consider holomorphic maps between compact Riemann surfaces as in the following diagram:
	\begin{equation}\label{eq:PairOfCoverings}
		\begin{tikzcd}
			& X \ar[dl,"f_1"'] \ar[dr, "f_2"]&\\[-1ex]
			X_1 \ar[dr,"g_1"'] & &  X_2 \ar[dl, "g_2"]\\[-1ex]
			& Y &
		\end{tikzcd}
	\end{equation}
	If \(Y\) is \emph{minimal} (namely, the maps \(g_1\) and \(g_2\) do not factor by a common intermediate covering but \(\Id_Y\)), then \(f_2^* \Prym(g_2)\) is an abelian subvariety of \(\Prym(f_1)\) (see \cite{paper:LangeRecillas04}*{Proposition~2.2}).
	\begin{defn}
		The \emph{Prym variety of the pair of coverings} \((f_1, f_2)\) is the complementary abelian subvariety of \(f_2^*\Prym(g_2)\) in \(\Prym(f_1)\), it is denoted by \(\Prym(f_1,f_2)\).
	\end{defn}
	By, if necessary, redefining \(X = X_1 \times_Y X_2\), we can assume that \(X\) is  \emph{minimal} (namely, the maps \(f_1\) and \(f_2\) do not factor by a common covering but \(\Id_X\)).
	If \(Y\) is also minimal, the following result yields the dimension of \(\Prym(f_1,f_2)\).
	\begin{prop}[\cite{paper:LangeRecillas04}*{Proposition~2.5}]\label{prop:DimForPrymOfPairs}
		For any pair of coverings as in diagram \eqref{eq:PairOfCoverings} where \(X\) and \(Y\) are minimal, we have
		\begin{align*}
		\dim \Prym(f_1,f_2) ={}& (d_1 - 1)(d_2 - 1)(g - 1) + \dfrac{1}{2}\bigl[\deg(R_{f_1}) + (d_1 - 1)\deg(R_{g_1}) - \deg(R_{g_2})\bigr]\text{,}
		\end{align*}
		where \(d_1\) and \(d_2\) are the degrees of \(f_1\) and \(f_2\), respectively.
	\end{prop}

	\subsubsection{Group algebra decomposition through Prym varieties}

	Since \(G\) acts naturally on \(\J(\hat X)\), there is an algebra homomorphism \(\QQ[G] \to \End_{\QQ}(\J(\hat X))\coloneqq \End(\J(\hat X)) \otimes \QQ\). For simplicity, we denote its images just as elements of \(\QQ[G]\) (althought it is not, in general, an isomorphism).
    For any \(\alpha \in \End_\QQ(\J(\hat X))\), we define \(\ima (\alpha)\coloneqq \ima(m\alpha)\) where \(m\) is any positive integer such that \(m\alpha \in \End(\J(\hat X))\); this definition certainly does not depend on \(m\).
	Set \(\Irr_\QQ(G) = \{W_1,\ldots,W_r\}\) with \(W_1 \cong 1_{G}\). The rational isotypical decomposition of \(\QQ[G]\), given by \cref{eq:QGdecomposition}, induces a decomposition of \(\J(\hat X)\) as follows directly from \cite{paper:LangeRecillas04(2)}*{Proposition~1.1}.

	\begin{prop}\hfill
		\begin{itemize}
			\item Each \(\ima e_i\) is \(G\)-stable, and \(\operatorname{Hom}_{G}(\ima e_i,\ima e_j) = 0\) whenever \(i\neq j\).

			\item The addition map induces an isogeny \(\ima e_1 \times \cdots \times \ima e_r \to \J(\hat X)\).
		\end{itemize}
	\end{prop}

	This decomposition is called the \emph{isotypical decomposition} of \(\J(\hat X)\) and it is unique up to a permutation of the factors. Recall the notation of \cref{defn:GroupAlgebraDecomposition}; the abelian subvarieties \(\ima q_{i,j}\) are pairwise isogenous for each fixed \(i \in \{1, \ldots, r\}\). Thereby, there are abelian subvarieties \(B_i\) such that \(\ima e_i\) is isogenous to \(B_i^{l_i}\) (we can set \(B_i \coloneqq \ima q_{i,1}\), for example, but this certainly does not determine \(B_i\) uniquely). Thereby, we have the following result (see \cite{paper:LangeRecillas04(2)}*{p.~140})

	\begin{theo}\label{theo:GroupAlgebraDecomposition}
        There are abelian subvarieties \(B_1, \ldots, B_r\) of \(\J(\hat X)\) and a \(G\)-equivariant isogeny
		\begin{equation*}
			\J(\hat X) \sim B_1^{l_1} \times \cdots \times B_r^{l_r}\text{,}
		\end{equation*}
		where \(l_i = \dim(V_i)/m_i\) with \(V_i\in\Irr_\CC(G)\) Galois associated to \(W_i\).
	\end{theo}
	This is called the \emph{group algebra decomposition} of \(\J(\hat X)\). If \(W_1 \cong 1_{G}\), then \(l_1 = 1\) and \(B_1 \sim \J(Y)\) (see \cite{paper:CR2006}). For the rest of the \(B_i\), we have the following result, which is a direct consequence of \cite{paper:CR2006}*{Corollary~5.6}.

	\begin{prop}\label{prop:GAComponentAsPrym}
        Let \(H\) and \(N\) be subgroups of \(G\) with \(H\subseteq N\). If \(\rho_H = W_i \oplus \rho_N\), then \(\Prym\bigl( \pi_H^N \bigr) \sim B_i\). 
	\end{prop}


	Now let \(N_1\) and \(N_2\) be two subgroups of \(\Mon(f)\). Set \(M \coloneqq N_1 \cap N_2\) and \(N \coloneqq \angp{N_1, N_2}\), then \(\hat X/M\) and \(\hat X/N\) are both minimal:
    In this case, we have the following result.

	\begin{theo}[\cite{paper:LangeRecillas04}*{Corollary~3.5}]\label{theo:PrymOfPairsOfCoverings}
        If each rational irreducible representation of \(G\) is absolutely irreducible and
		\(
			\chi_M + \chi_N - \chi_{N_1} - \chi_{N_2} = \sum_{i=2}^{r} s_i \chi_{W_i}\text{,}
		\) then \(\Prym(\pi_{N_1}^M,\pi_{N_2}^M) \sim B_2^{s_2} \times \cdots \times B_r^{s_r}\).
	\end{theo}

	\section{Galois closure of a 5-fold covering}\label{sec:MonodromyOfF}

	In this section, we give necessary and sufficient criteria for a potential ramification data of \(f\) to be actually realizable. Then, all possible monodromy groups \(G\) modulo conjugacy in \(\Sym_5\) are tabulated in terms of those realizable tuples of ramification data.

	\subsection{Realizable ramification data}

    The type of each branch value of \(f\) is the cycle structure of a permutation in \(\Sym_5\); hence, it can be \([5]\), \([4,1]\), \([3,2]\), \([3,1,1]\), \([2,2,1]\) or \([2,1,1,1]\). This possible types are classified in two classes:
	\begin{itemize}
		\item the even ones, namely \([5]\), \([3,1,1]\) and \([2,2,1]\); and
		\item the odd ones, namely \([4,1]\), \([3,2]\) and \([2,1,1,1]\).
	\end{itemize}
    \begin{rema}
        Each ramification type is assumed to be the cycle structure of a permutation in \(\Sym_5\).
    \end{rema}

	\begin{defn}
		A tuple of ramification types \((t_1\ldots,t_n)\) with an even number of odd branch values will be called \emph{even}.
	\end{defn}

	\begin{theo}\label{theo:RealizablePermutationGP}
        A \(5\)-fold covering \(f\colon X\to Y\) with \(g\geq 1\) and ramification data \((t_1, \ldots, t_n)\) exists if and only if \((t_1\ldots,t_n)\) is even.
	\end{theo}

	\begin{proof}
		Sufficiency of the hypothesis is a consequence of the Riemann--Hurwitz formula because \(\ord{R_f}\) is even if and only if there is an even number of odd branch values.

		To prove the necessity of the hypothesis, an explicit construction will be given.
        According to \cref{prop:ExistenceOfHolomorphicMaps}, giving a \(5\)-fold covering \(f\colon X\to Y\) with branch values \(y_1,\ldots, y_n\) of types \(t_1,\ldots,t_n\) is equivalent to giving a generating \((g;m_1,\ldots,m_n)\)-vector \((a_1, b_1, \ldots, a_{g}, b_{g}, c_1, \ldots, c_n)\) of a transitive subgroup \(G\) of \(\Sym_5\) such that each \(c_i\) has cycle structure \(t_i\).

		For the case where \(n = 0\), since \(g \geq 1\), we have the generating \((g;)\)-vector
        \(((1\,2\,3\,4\,5), \Id_{\times(2g - 1)})\)
		of the transitive subgroup \(\angp{(1\,2\,3\,4\,5)}\) of \(\Sym_5\), where \(\Id_{\times(2g - 1)}\) denotes a list with \(\Id\) repeated \(2g - 1\) times (we will keep this notation throughout this article).

		For the case where \(n > 0\), choose each permutation \(c_i\) of cycle structure \(t_i\) arbitrarily. The permutation \(\prod_{i=1}^{n}c_i\) is even because there is an even number of odd permutations \(c_i\), so it is of type \([5]\), \([3,1,1]\) or \([2,2,1]\); thereby, up to conjugacy in \(\Sym_5\), we can be assume that \(\prod_{j=1}^{n}c_i\) is \((1\,2\,3\,4\,5)\), \((1\,2\,3)\) or \((1\,2)(3\,4)\). Set \(a_i = b_i = \Id\) for \(i \in \{2,\ldots, g\}\), and, depending on \(\prod_{i=1}^n c_i\), set \(a_1\) and \(b_1\) according to \cref{table:ExistanceOfCoveringsOfDegree5}. Thereby, \((a_1,\ldots,a_{g},b_1,\ldots,b_{g},c_1,\ldots,c_n)\) is a generating \((g;m_1,\ldots, m_n)\)-vector of a subgroup \(G\) of \(\Sym_5\); we just need to prove that \(G\) is transitive, but:
		\begin{itemize}
			\item If \(\prod_{i=1}^n c_i = (1\,2\,3\,4\,5)\), then \((1\,2\,3\,4\,5) \in G\) and hence \(G\) transitive.
			\item If \(\prod_{i=1}^n c_i = (1\,2)(3\,4)\) or \(\prod_{i=1}^n c_i = (1\,2\,3)\), then \(a_1b_1\) is of type \([5]\) and hence \(G\) is transitive.\qedhere
		\end{itemize}


		\begin{table}
			\caption{Choice of permutations for a generating vector of a transitive subgroup of \(\Sym_5\)}\label{table:ExistanceOfCoveringsOfDegree5}
			\begin{tabular}{@{} *4{>{\(} l <{\)}} @{}}\toprule
				\prod_{i=1}^n c_i	& a_1				& b_1			& \prod_{i=1}^{g_Y}[a_i,b_i] \\ \midrule
				(1\,2\,3\,4\,5)		& (1\,3)(2\,5\,4)	& (1\,3\,5\,2)	& (1\,5\,4\,3\,2)\\
				(1\,2)(3\,4)		& (1\,3)(2\,4\,5)	& (1\,2\,4\,5)	& (1\,2)(3\,4)\\
				(1\,2\,3)			& (1\,5\,2)(3\,4)	& (1\,3\,5\,4)	& (1\,3\,2)\\ \bottomrule
			\end{tabular}
		\end{table}
	\end{proof}

	\begin{coro}\label{coro:SymmetricIsRealizable}
        If \(n>0\), the map \(f\colon X\to Y\) of \cref{theo:RealizablePermutationGP} may be chosen such that \(G =\Sym_5\).
	\end{coro}
	\begin{proof}
		From the proof of \cref{theo:RealizablePermutationGP}, we see that, for every even tuple \((t_1\ldots,t_n)\) with \(n>0\), we can choose a generating vector such that \(G\) contains permutations of order \(5\), \(6\) and \(4\) (see \cref{table:ExistanceOfCoveringsOfDegree5}). Since \(\lcm(4,5,6) = 60\), we have \(60\mid\ord{G}\); hence \(G \in \{\Alt_5, \Sym_5\}\). Since \(a_1\) is an odd permutation, we have \(G = \Sym_5\).
	\end{proof}

	The situation becomes a bit more complicated when \(g = 0\); namely, when \(Y \cong \PP^1\). By the Riemann--Hurwitz formula, we have \(2g_X - 2 = 5(2g - 2) + \deg(R_f)\); but \(g = 0\), so \(2g_X = -8 + \deg(R_f)\) and, since \(g_X \geq 0\), we get \(8 \leq \deg(R_f)\). We will see that this condition is the only additional condition needed for \((t_1,\ldots, t_n)\) to be realizable.

	For \(t_i = [\nu_{i, 1},\ldots, \nu_{i, {k_i}}]\), we set \(\deg(t_i) \coloneqq \sum_{j=1}^{k_i} (\nu_{i,j} - 1)\) and \(\deg(t_1,\ldots,t_n) \coloneqq \sum_{i=1}^{n} \deg(t_i)\).
	We call \(\deg(t_1,\ldots, t_n)\) the \emph{degree} of the tuple \((t_1,\ldots, t_n)\) since it coincides with \(\deg(R_f)\) when \((t_1,\ldots, t_n)\) is the actual ramification data of a map \(f\).

	\begin{theo}\label{theo:RealizablePermutationG0}
        There is a \(5\)-fold covering map \(f\colon X\to \PP^1\) with branch data \((t_1, \ldots, t_n)\) if and only if the tuple \((t_1,\ldots,t_n)\) is even and \(\deg(t_1,\ldots,t_n) \geq 8\).
	\end{theo}

    In order to prove \cref{theo:ClassificationG0}, we will need \cref{lemm:AnyCycleByAProductInA5,lemm:Type5CycleAsProduct}.

	\begin{lemm}\label{lemm:AnyCycleByAProductInA5}
		Consider an even tuple \((t_1,\ldots, t_n)\) with \(n \geq 2\) and at least one even permutation. For each even cycle structure \(t\), there are permutations \(c_1, \ldots, c_n\) in \(\Sym_5\) such that each \(c_i\) is of type \(t_i\) and \(\prod_{i=1}^n c_i\) is of type \(t\).
	\end{lemm}
	\begin{proof}
		Without loss of generality, assume that \(t_1\) is even. For each \(i\in \{2,\ldots,n\}\), chose any \(c_i \in \Sym_5\) of type \(t_i\). Since \((t_1,\ldots, t_n)\) is even, the permutation \(\prod_{i=2}^n c_i\) is also even; thus, modulo conjugation, we can assume that \(\prod_{i=2}^n c_i\) is \((1\,2\,3\,4\,5)\), \((1\,2)(3\,4)\) or \((1\,2\,3)\). Choose \(c_1\) depending on \(t\), \(t_1\) and \(\prod_{i=2}^{n} c_i\) according to \cref{table:ChoiceOfC1}. Thereby, the permutation \(\prod_{i=1}^{n} c_i\) is of type \(t\) and each \(c_i\) of type \(t_i\). \qedhere

		\begin{table}
			\caption{Choice of permutations with prescribed cycle structure of its product}\label{table:ChoiceOfC1}
			\begin{tabular}{@{} *5{>{\(} l <{\)}} @{}}\toprule
				t		& \prod_{i = 2}^{n} c_i	& t_1	& c_1	& \prod_{i = 1}^{n} c_i\\\midrule
				\multirow{9}{*}{[5]} & \multirow{3}{*}{(1\,2\,3\,4\,5)} & [5] & (1\,2\,3\,4\,5) & (1\,3\,5\,2\,4)\\
				&& [2,2,1]	& (1\,3)(2\,4) 	& (1\,4\,5\,3\,2)\\
				&& [3,1,1]	& (1\,2\,5)	& (1\,5\,2\,3\,4)\\ \cmidrule(l){3-5}
				& \multirow{3}{*}{(1\,2)(3\,4)}	& [5]	& (1\,3\,2\,4\,5)	& (1\,4\,2\,3\,5)\\
				&& [2,2,1]	& (1\,5)(2\,3)	& (1\,3\,4\,2\,5)\\
				&& [3,1,1]	& (1\,5\,3)	& (1\,2\,5\,3\,4)\\ \cmidrule(l){3-5}
				& \multirow{3}{*}{(1\,2\,3)}	& [5]	& (1\,2\,5\,4\,3)	& (1\,5\,4\,3\,2)\\
				&& [2,2,1]	& (1\,5)(3\,4)	& (1\,2\,4\,3\,5)\\
				&& [3,1,1]	& (2\,4\,5)	& (1\,4\,5\,2\,3)\\ \midrule
				\multirow{9}{*}{[2,2,1]} & \multirow{3}{*}{(1\,2\,3\,4\,5)} & [5] & (1\,3\,2\,4\,5) & (1\,4)(3\,5)\\
				&& [2,2,1]	& (1\,5)(2\,4) 	& (1\,4)(2\,3)\\
				&& [3,1,1]	& (1\,4\,3)	& (1\,2)(4\,5)\\ \cmidrule(l){3-5}
				& \multirow{3}{*}{(1\,2)(3\,4)}	& [5]	& (1\,4\,5\,3\,2)	& (2\,4)(3\,5)\\
				&& [2,2,1]	& (1\,4)(2\,3)	& (1\,3)(2\,4)\\
				&& [3,1,1]	& (1\,2\,5)	& (1\,5)(3\,4)\\ \cmidrule(l){3-5}
				& \multirow{3}{*}{(1\,2\,3)}	& [5]	& (1\,4\,3\,2\,5)	& (1\,5)(3\,4)\\
				&& [2,2,1]	& (2\,3)(4\,5)	& (1\,3)(4\,5)\\
				&& [3,1,1]	& (1\,5\,3)	& (1\,2)(3\,5)\\ \midrule
				\multirow{9}{*}{[3,1,1]} & \multirow{3}{*}{(1\,2\,3\,4\,5)} & [5] & (1\,4\,5\,3\,2) & (3\,5\,4)\\
				&& [2,2,1]	& (1\,5)(3\,4) 	& (1\,2\,4)\\
				&& [3,1,1]	& (1\,5\,4)	& (1\,2\,3)\\ \cmidrule(l){3-5}
				& \multirow{3}{*}{(1\,2)(3\,4)}	& [5]	& (1\,5\,4\,3\,2)	& (2\,5\,4)\\
				&& [2,2,1]	& (1\,5)(3\,4)	& (1\,2\,5)\\
				&& [3,1,1]	& (1\,3\,4)	& (1\,2\,3)\\ \cmidrule(l){3-5}
				& \multirow{3}{*}{(1\,2\,3)}	& [5]	& (1\,5\,4\,3\,2)	& (3\,5\,4)\\
				&& [2,2,1]	& (1\,5)(2\,3)	& (1\,3\,5)\\
				&& [3,1,1]	& (1\,2\,3)	& (1\,3\,2)\\ \bottomrule
			\end{tabular}
		\end{table}
	\end{proof}

	\begin{lemm}\label{lemm:Type5CycleAsProduct}
		Consider an even tuple \((t_1,\ldots,t_n)\) of cycle structures of permutations in \(\Sym_5\) such that \(\deg(t_1,\ldots,t_n) \geq 4\). There are permutations \(c_1, \ldots, c_n\) in \(\Sym_5\) such that each \(c_i\) is of type \(t_i\) and \(\prod_{i=1}^n c_i\) is of type \([5]\).
	\end{lemm}

	\begin{proof}
		Suppose there is at least one even permutation in \((t_1,\ldots,t_n)\), say \(t_1\). Since \(\deg(t_1,\ldots,t_n) \geq 4\), if \(n = 1\), then \(t_1 = [5]\) and the lemma is trivially satisfied. If \(n \geq 2\), then \cref{lemm:AnyCycleByAProductInA5} implies the existence of the permutations \(c_i\).

		Now suppose that there are no even permutations in \((t_1,\ldots,t_n)\), so \(n \geq 2\). Suppose that at least one \(t_i\), namely \(t_1\), is of type \([4,1]\) or \([3,2]\). Set \(c_i\) as any permutation with cycle structure \(t_i\) for \(i\in \{2,\ldots, n\}\). Since \((t_1,\ldots,t_n)\) is even, the product \(\prod_{i=2}^{n} c_i\) must be an odd permutation; hence, its cycle structure is \([2,1,1,1]\), \([4,1]\) or \([3,2]\). Thereby, we can assume that \(\prod_{i=2}^{n} c_i\) is \((1\,2)\), \((1\,2\,3\,4)\) or \((1\,2\,3)(4\,5)\); for each case, choose \(t_1\) according to its type as given in \cref{table:ChoiceOfC1Odds}. In this manner, the product \(\prod_{i=1}^{n} c_i\) is of type \([5]\) and each \(c_i\) has cycle structure \(t_i\).

		\begin{table}
			\caption{Product of odd permutations with cycle structure \([5]\)}\label{table:ChoiceOfC1Odds}
			\begin{tabular}{@{} *4{>{\(} l <{\)}} @{}}\toprule
				\prod_{i = 2}^{n} c_i	& t_1	& c_1	& \prod_{i = 1}^{n} c_i\\\midrule
				\multirow{2}{*}{(1\,2)} & [3,2] & (1\,5\,3)(2\,4) & (1\,4\,2\,5\,3)\\
				& [4,1]	& (1\,5\,3\,4) 	& (1\,2\,5\,3\,4)\\ \cmidrule(l){2-4}
				\multirow{2}{*}{(1\,2\,3\,4)}	& [3,2] & (1\,3\,4)(2\,5) & (1\,5\,2\,4\,3)\\
				& [4,1]	& (2\,3\,4\,5) 	& (1\,3\,5\,2\,4)\\ \cmidrule(l){2-4}
				\multirow{2}{*}{(1\,2\,3)(4\,5)}& [3,2] & (1\,5\,3)(2\,4) & (1\,4\,3\,5\,2)\\
				& [4,1]	& (1\,5\,2\,3) 	& (1\,3\,5\,4\,2)\\ \bottomrule
			\end{tabular}
		\end{table}

		Finally, suppose that there are neither even nor type \([4,1]\) or \([3,2]\) permutations in \((t_1, \ldots, t_n)\); that is, we have \(t_i = [2,1,1,1]\) for each \(i \in \{1,\ldots,n\}\). Since \(\deg(t_1,\ldots,t_n) \geq 4\), we have that \(n \geq 4\). Set \(c_1 = (3\,4)\), \(c_2 = (2\,3)\) and \(c_3 = (1\,2)\); thus \(\prod_{i = 1}^3 c_i = (1\,2\,3\,4)\). Applying the previous paragraph to the tuple \(([4,1],t_4, \ldots, t_n)\) yields permutations \(c_4,\ldots,c_n\) of type \([2,1,1,1]\) such that \(\prod_{i=1}^{n} c_i\) is of type \([5]\) and each \(c_i\) has cycle structure \(t_i\).
	\end{proof}

	\begin{proof}[Proof of \cref{theo:RealizablePermutationG0}]
		The sufficiency of the hypothesis is a direct consequence of the Riemann--Hurwitz formula.

		To prove necessity, we give explicit constructions in the several possible cases.
        According to \cref{prop:ExistenceOfHolomorphicMaps}, giving a \(5\)-fold covering \(f\colon X\to \PP^1\) with branch types \(t_1,\ldots,t_n\) is equivalent to give a generating \((0;m_1,\ldots,m_n)\)-vector \((c_1, \ldots, c_n)\) of a transitive subgroup of \(\Sym_5\) such that each \(c_i\) has cycle structure \(t_i\).

		Suppose that we are in one of the following four situations:
        \begin{enumerate}
			\item\label{item:con1} there is a \(t_i\) of degree \(4\), namely \([5]\), in \((t_1,\ldots,t_n)\);
			\item\label{item:con2} there are at least two \(t_i\) of degree \(2\), namely \([3,1,1]\) or \([2,2,1]\);
			\item\label{item:con3} there are at least one \(t_i\) of degree \(3\), namely \([4,1]\) or \([3,2]\), and one of degree \(1\), namely \([2,1,1,1]\); or
			\item\label{item:con4} there are at least four \(t_i\) of degree \(1\).
		\end{enumerate}
		In any of these four cases, probably after a re-enumeration, we can take an even 
        sub-tuple \((t_1,\ldots,t_k)\) of \((t_1,\ldots, t_n)\) with \(\deg(t_1,\ldots,t_k) = 4\); hence, according to \cref{lemm:Type5CycleAsProduct}, we can choose permutations \(c_1,\ldots, c_k\) of types \(t_1, \ldots, t_k\), respectively, such that \(\prod_{i=1}^{k} c_i= (1\,2\,3\,4\,5)\). Since \(\deg(t_1,\ldots, t_n)\geq 8,\) we have that \(\deg(t_{k+1},\ldots,t_n) = \deg(t_1,\ldots, t_n) - \deg(t_1,\ldots,t_k) \geq 4;\) hence, again by \cref{lemm:Type5CycleAsProduct}, we can choose \(c_{k+1}, \ldots, c_n\) with cycle structures \(t_{k+1}, \ldots, t_n\), respectively, such that \(\prod_{i=k+1}^{n} c_i = (1\,5\,4\,3\,2)\). 

		If none of the conditions \eqref{item:con1} to \eqref{item:con4} are satisfied, then every \(t_i\) must have degree~\(3\) except maybe for one cycle structure of degree \(2\), say \(t_n\). Since \(\deg(t_1,\ldots,t_n)\geq 8\), we have \(n\geq 3\). If \(n \geq 4\), then \cref{lemm:Type5CycleAsProduct} can be applied separately to the sub-tuples \((t_1,t_2)\) and \((t_3,\ldots, t_n)\) yielding permutations \(c_1,\ldots,c_n\) of types \(t_1,\ldots,t_n\), respectively, such that \(c_1c_2 = (1\,2\,3\,4\,5)\) and \(\prod_{i=3}^{n} c_i = (1\,5\,4\,3\,2)\). 
        Otherwise, we have \(n = 3\), so \(\deg(t_1) = \deg(t_2) = 3\) and \(\deg(t_3) = 2\). If there is at least one \([3,2]\) cycle structure in \((t_1,t_2,t_3)\) set it as \(t_1\). Choose \(c_1\) to be \((1\,2\,3)(4\,5)\) or \((1\,2\,3\,4)\) according to \(t_1\) and then choose \(c_2\) according to \(t_3\) and \(c_1\) as given in \cref{table:ChoiceOfCiCase332}; set \(c_3 \coloneqq c_2^{-1}c_1^{-1}\). Thereby, in each sub-case, the permutation \(c_i\) has cycle structure \(t_i\) for each \(i\in\{1,2,3\}\) and \(c_1c_2c_3 = \Id\). Moreover, column \([5]\) of \cref{table:ChoiceOfCiCase332} gives a permutation of type \([5]\) that belongs to \(\angp{c_1,c_2,c_3}\) for each of the six sub-cases; therefore, the tuple \((c_1,c_2,c_3)\) is a generating vector of a transitive subgroup of \(\Sym_5\). \qedhere

		\begin{table}
			\caption{Choice of odd permutations with product of degree-\(2\) cycle structure}\label{table:ChoiceOfCiCase332}
			\begin{tabular}{@{} *6{>{\(} l <{\)}} @{}}\toprule
				t_3		& c_1	& t_2	& c_2	& c_1c_2 & [5] \\\midrule
				\multirow{3}{*}{[2,2,1]} & \multirow{2}{*}{(1\,2\,3)(4\,5)} & [3,2] & (1\,4\,2)(3\,5) & (1\,5)(3\,4) & c_2^{-1} c_1 c_2^{-2} c_1^2\\
				&& [4,1]	& (2\,3\,4\,5) 	& (1\,2)(3\,5) & c_1^{-2} c_2^{-1} c_1^3 c_2 c_1^{-1} c_1^{-2}\\ \cmidrule(l){3-6}
				& (1\,2\,3\,4)	& [4,1]	& (2\,4\,5\,3)	& (1\,2)(4\,5) & c_1 c_2^{-1}\\\midrule
				\multirow{3}{*}{[3,1,1]} & \multirow{2}{*}{(1\,2\,3)(4\,5)} & [3,2] & (1\,4\,5)(2\,3) & (1\,5\,2) & c_2 c_1^2 c_2^{3} \\
				&& [4,1]	& (2\,5\,4\,3) 	& (1\,2\,4) & c_1^{-1} c_2^{-1} c_1^{-2} c_2^{-1} c_1^3 \\ \cmidrule(l){3-6}
				& (1\,2\,3\,4)	& [4,1]	& (2\,5\,4\,3)	& (1\,2\,5) & c_1^{-1} c_2 c_1^{-1} c_2\\ \bottomrule
			\end{tabular}
		\end{table}
	\end{proof}

	\subsection{Monodromy group in terms of the ramification data}\label{sec:MonInTermsOfRamificationData}

	Now we give necessary and sufficient conditions on \((t_1,\ldots,t_n)\) for a transitive subgroup \(G\) of \(\Sym_5\) to be the actual monodromy group of a \(5\)-fold covering \(f\colon X\to Y\); these criteria will depend on \(g\).

	\begin{prop}\label{prop:PossibleMonodromy}
		The monodromy group \(G\) is conjugate to one of the following subgroups of \(\Sym_5\):
		\begin{enumerate}
			\item the \emph{cyclic group} \(\angp{(1\, 2\, 3\, 4\, 5)}\), denoted by \(\Cyc_5\);

			\item the \emph{dihedral group} \(\angp{(1\,2\,3\,4\,5), (2\,5)(3\,4)}\), denoted by \(\Dih_5\);

			\item\label{item:GroupGA} the group \(\angp{(1\,2\,3\,4\,5), (2\,3\,5\,4)}\), isomorphic to the \emph{general affine group} \(\GA\) (it will be denoted just by \(\GA\));

			\item the \emph{alternating group} \(\angp{(1\,2\,3\,4\,5), (1\,3\,4\,5\,2)}\), denoted by \(\Alt_5\); or

			\item the \emph{symmetric group} \(\angp{(1\,2\,3\,4\,5), (1\,2)}\), denoted by \(\Sym_5\).
		\end{enumerate}
	\end{prop}
	\begin{proof}
		The monodromy group of \(f\) is a transitive subgroup of \(\Sym_5\); those subgroups and their generators are tabulated in \cite{paper:TransitiveGroups}*{Tables~5A~and~5B}.
	\end{proof}
	The cycle structure of the elements of each group listed in \cref{prop:PossibleMonodromy} are tabulated in \cite{paper:TransitiveGroups}*{Table~5C}; \cref{table:CycleStructure} summarizes those cycle structures.
%


	\begin{table}
		\caption{Cycle structure of permutations in each possible \(G\)}
		\label{table:CycleStructure}
		\begin{tabular}{@{}ll@{}}\toprule
			Monodromy group \(G\) up to isomorphism & Cycle structure of permutations in \(G\)\\
			\midrule
			Cyclic group \(\Cyc_5\) &  \([5]\) \\
			Dihedral group \(\Dih_5\) &  \([5]\) or \([2,2,1]\) \\
			Affine group \(\GA\) & \([5]\), \([4,1]\) or \([2,2,1]\) \\
			Alternating group \(\Alt_5\) & \([5]\), \([3,1,1]\) or \([2,2,1]\) \\
			Symmetric group \(\Sym_5\) & \([5]\), \([4,1]\), \([3,2]\), \([3,1,1]\), \([2,2,1]\) or \([2,1,1,1]\)\\ \bottomrule
		\end{tabular}
	\end{table}

	In order to give a straightforward proof of \cref{theo:ClassificationGP}, we establish some auxiliary lemmata.

	\begin{lemm}\label{lemm:ProdsInD5}
		For a nonempty finite subset \(\{c_1,\ldots,c_n\}\) of permutations in \(\Dih_5\) with exactly \(m\) permutations of cycle structure \([2,2,1]\), we have:
		\begin{enumerate}
			\item If \(m\) is odd, then \(\prod_{i=1}^{n} c_i\) is of type \([2,2,1]\).
			\item If \(m\) is even, then \(\prod_{i=1}^{n} c_i\) is of type \([5]\) or the identity.
		\end{enumerate}
	\end{lemm}
	\begin{proof}
		Consider the quotient homomorphism \(\phi\colon \Dih_5 \to \Dih_5/\Cyc_5\). The image by \(\phi\) of any type \([5]\) permutation is the identity, while the image of a type \([2,2,1]\) permutation is the only non-trivial element \((1\,2)(3\,4)\Cyc_5\) of \(\Dih_5/\Cyc_5\); so \(\phi(\prod_{i=1}^{n} c_i) = \prod_{i=1}^{n} \phi(c_i) ((1\,2)(3\,4))^{m} \Cyc_5\). Hence, whether or not \(\prod_{i=1}^{n} c_i\) belongs to \(\Cyc_5\) depends solely on the parity of \(m\).
	\end{proof}

	\begin{lemm}\label{lemm:ProdsInD5Even221}
		For a nonempty tuple \((t_1,\ldots,t_n)\) of cycle structures \([2,2,1]\) or \([5]\) with an even number of \(t_i\) equal to \([2,2,1]\), we can choose a permutation \(c_i\) in \(\Dih_5\) of type \(t_i\) for each \(i\in \{1,\ldots, n\}\) such that \(\prod_{i = 1}^n c_i\) is of type \([5]\).
	\end{lemm}

	\begin{proof}
		We proceed by induction. If \(n=1\), set \(c_1 \coloneqq (1\,2\,3\,4\,5)\). If \(n = 2\), we have two cases: If there are permutations of type \([2,2,1]\), then \(t_1 = t_2 = [2,2,1]\), so we set \(c_1 \coloneqq (2\,5)(3\,4)\) and \(c_2 \coloneqq (1\,5)(2\,4)\), hence \(c_1c_2 = (1\,2\,3\,4\,5)\). If there are no permutations of type \([2,2,1]\), then \(t_1 = t_2 = [5]\), so we set \(c_1 \coloneqq (1\,2\,3\,4\,5)\) and \(c_2 \coloneqq (1\,2\,3\,4\,5)\), hence \(c_1c_2 = (1\,3\,5\,2\,4)\).

		Now consider \(n \geq 3\). If there is at least one type \([5]\) in \((t_1,\ldots,t_n)\), say \(t_n\), by inductive hypothesis, we can choose permutations \(c_i\) of type \(t_i\) for \(i\in \{1,\ldots, n - 1\}\) such that \(\prod_{i = 1}^{n-1} c_i\) is of type \([5]\), say \(\prod_{i = 1}^{n-1} c_i = (1\,2\,3\,4\,5)\). Set \(c_{n} \coloneqq (1\,2\,3\,4\,5)\), so \(\prod_{i = 1}^n c_i = (1\,3\,5\,2\,4)\). If there are no permutations of type \([5]\), by inductive hypothesis, we can choose permutations \(c_i\) of type \(t_i\) for \(i\in \{1,\ldots, n - 2\}\) such that \(\prod_{i = 1}^{n-2} c_i\) is of type \([5]\), say \(\prod_{i = 1}^{n-2} c_i = (1\,2\,3\,4\,5)\). So we set \(c_{n-1} \coloneqq (2\,5)(3\,4)\) and \(c_{n} \coloneqq (1\,5)(2\,4)\), hence \(\prod_{i = 1}^{n} c_i = (1\,3\,5\,2\,4)\).
	\end{proof}

	\begin{theo}\label{theo:ClassificationGP}
		Consider a \(5\)-fold covering \(f\colon X\to Y\) with ramification data \((t_1,\ldots, t_n)\) and monodromy group \(G\). If \(g \geq 1\), then the following statements hold:
		\begin{enumerate}
			\item\label{item:GPN=0} Suppose that \(n = 0\). If \(g = 1\), then \(G \cong \Cyc_5\). If \(g > 1\), then \(G\) may be conjugate to any group in \cref{prop:PossibleMonodromy}.

			\item\label{item:GP32or2111} If at least one \(t_i\) equals \([3,2]\) or \([2,1,1,1]\), then \(G = \Sym_5\).

			\item\label{item:GP311and41} If there are types \([3,1,1]\) and \([4,1]\) in \((t_1,\ldots,t_n)\), then \(G = \Sym_5\).

			\item\label{item:GPS5orGA} If there are no types \([3,2]\), \([3,1,1]\) or \([2,1,1,1]\) in \((t_1,\ldots, t_n)\), but at least one \(t_i\) equals \([4,1]\), then \(G = \Sym_5\) or \(G \cong \GA\).

			\item\label{item:GPS5orA5} If \((t_1,\ldots, t_n)\) has only even types and at least one of them is \([3,1,1]\), then \(G \cong \Sym_5\) or \(G \cong \Alt_5\).

			\item\label{item:GP221} Suppose that there are neither odd nor \([3,1,1]\) types in \((t_1,\ldots, t_n)\), but at least one \(t_i\) equals \([2,2,1]\). We have three cases:
			\begin{enumerate}
				\item\label{item:GP221a} If \(n = 1\), \(t_1 = [2,2,1]\) and \(g = 1\); then \(G = \Sym_5\).

				\item\label{item:GP221b} If there is an odd number of types \([2,2,1]\) in \((t_1,\ldots,t_n)\), and \(n \geq 2\) or \(g \geq 2\); then \(G = \Sym_5\) or \(G = \Alt_5\).

				\item\label{item:GP221c} If there is an even number of types \([2,2,1]\) in \((t_1,\ldots,t_n)\), then \(G\) may be conjugate to any group in \cref{prop:PossibleMonodromy} but \(\Cyc_5\).
			\end{enumerate}

			\item\label{item:GPonly5} Suppose that \(t_i = [5]\) for each \(i\). There are two cases:
			\begin{enumerate}
				\item\label{item:GPonly5a} If \(n = 1\), then \(G\) may be conjugate to any group in \cref{prop:PossibleMonodromy} but \(\Cyc_5\)

				\item\label{item:GPonly5b} If \(n > 1\), then \(G\) may be conjugate to any group in \cref{prop:PossibleMonodromy}.
			\end{enumerate}
		\end{enumerate}

		Conversely, for any realizable tuple \((t_1,\ldots,t_n)\) of ramification types and for each possible \(G\) listed in \cref{item:GPN=0,item:GP32or2111,item:GP311and41,item:GPS5orGA,item:GPS5orA5,item:GP221,item:GPonly5}, there is a \(5\)-fold covering with that ramification data and that monodromy group.
	\end{theo}

	\begin{proof}
		According to \cref{prop:ExistenceOfHolomorphicMaps}, the holomorphic map \(f\colon X\to Y\) with ramification data \((t_1,\ldots, t_n)\) exists if and only if the group \(G\) has a generating \((g;m_1,\ldots, m_n)\)-vector \linebreak \((a_1,b_1, \ldots a_{g},b_{g}, c_1, \ldots, c_n)\) such that \(c_i\) has cycle structure \(t_i\) for each \(i\in\{1,\ldots,n\}\). In particular, each \(c_i\) belongs to \(G\); therefore, \(G\) contains permutations of each type \(t_i\).  We will separately prove \cref{item:GPN=0,item:GP32or2111,item:GP311and41,item:GPS5orGA,item:GPS5orA5,item:GP221,item:GPonly5} by assuming the existence of such a generating vector. After the proof of each item, we show the existence of a map \(f\) with the stated ramification data and monodromy group; usually, by explicitly constructing a generating vector with suitable properties. In the cases where there is only one possible monodromy group \(G\) up to conjugacy, the map \(f\) is given by \cref{theo:RealizablePermutationGP}; hence, the existence proof is omitted. Similarly, when \(G = \Sym_5\), existence of \(f\) is given by \cref{coro:SymmetricIsRealizable}, so the existence proof is also omitted.

		\textsc{\Cref{item:GPN=0}.} If \(g = 1\), then \(\fg(Y,y)\) is abelian; hence, the monodromy group \(G\) is also abelian. Therefore, by \cref{prop:PossibleMonodromy}, it must be conjugate to \(\Cyc_5\).

		\textsc{\Cref{item:GP32or2111}.} \Cref{table:CycleStructure} shows that the only transitive subgroup of \(\Sym_5\) that contains permutations of type \([3,2]\) or \([2,1,1,1]\) is \(\Sym_5\); hence \(G = \Sym_5\).

		\textsc{\Cref{item:GP311and41}.} The only group in \cref{table:CycleStructure} that contains permutations with both cycle structures, \([3,1,1]\) and \([4,1]\), is \(\Sym_5\); hence \(G = \Sym_5\).

		\textsc{\Cref{item:GPS5orGA}.} As in the two above items, according to \cref{table:CycleStructure}, we have \(G = \Sym_5\) or \(G \cong \GA\). For existence in the case where \(G \cong \GA\), we assume (modulo conjugation) that \(G = \GA\) and that \(c_1\) and \(c_2\) have cycle structure \([4,1]\). 
        Set \(c_i\) as an arbitrary permutation of type \(t_i\) in \(\GA\) for each \(i \in\{3,\ldots,n\}\); since \(\GA \cap \Alt_5 = \Dih_5\), we have that \(\prod_{i = 3}^n c_i\) equals, modulo conjugation, to \(\Id\), \((1\,4)(2\,3)\) or \((1\,2\,3\,4\,5)\). Set \(c_1\) and \(c_2\) as the permutations corresponding to the affine maps \(2x\) and \(3x\); \(2x\) and \(2x\); or \(2x+1\) and \(3x\) according to each possible \(\prod_{i = 3}^n c_i\) mentioned above. In this manner \(\prod_{i=1}^{n} c_i = \Id\) and \(\GA\) has the generating \((g;m_1,\ldots,m_n)\)-vector
		\(((1\,2\,3\,4\,5),\Id_{\times (2g - 1)},\allowbreak c_1, \ldots, c_n)\).

		\textsc{\Cref{item:GPS5orA5}.} According to \cref{table:CycleStructure}, we have \(G = \Sym_5\) or \(G = \Alt_5\).

		Now we prove existence in the case where \(G = \Alt_5\). By \cref{lemm:AnyCycleByAProductInA5}, if \(n \geq 2\) we can choose a permutation \(c_i \in \Alt_5\) of type \(t_i\) for each \(i \in \{1,\ldots, n\}\) such that \(\prod_{i=1}^{n} c_i\) is of type \([3,1,1]\); since there is at least one permutation of type \([3,1,1]\), we can still chose the \(c_i\) in the above manner if \(n = 1\). Without loss of generality, suppose \(\prod_{i=1}^{n} c_i = (1\,3\,2)\). Since \(\angp{(1\,5\,2\,3\,4),(1\,3\,4\,2\,5)} = \Alt_5\) 
        and \([(1\,5\,2\,3\,4),(1\,3\,4\,2\,5)] = (1\,2\,3)\), the group \(\Alt_5\) has the suitable generating \((g;m_1,\ldots,m_n)\)-vector
		\(
		((1\,5\,2\,3\,4), (1\,3\,4\,2\,5),\Id_{\times (2g - 2)},c_1,\ldots,c_n).
		\)

		\textsc{\Cref{item:GP221}.} According to \cref{table:CycleStructure}, the monodromy group \(G\) could be isomorphic to any transitive subgroup of \(\Sym_5\) but \(\Cyc_5\), which implies \cref{item:GP221c}. 
        Suppose that there is an odd number of branch values of type \([2,2,1]\) and \(G\) is \(\Dih_5\) or \(\GA\) modulo conjugacy. 
        Let \((a_1,b_1,\ldots,a_g, b_{g}, c_1,\dots,c_n)\) be a generating vector of \(G\); thereby, the permutation \(c_i\) is of type \([5]\) or \([2,2,1]\) for each \(i\in \{1,\ldots,n\}\). 
        \Cref{lemm:ProdsInD5} and \(\GA \cap \Alt_5 = \Dih_5\) implies that \(\prod_{i=1}^{n} c_i\) is of type \([2,2,1]\). On the other hand, we have \(\GA' = \Dih_5' = \Cyc_5\); hence, by \cref{eq:genVector}, there is a permutation of type \([2,2,1]\) in \(\Cyc_5\), which contradicts \cref{table:CycleStructure}. That contradiction proves \cref{item:GP221b}. 
        If \(n=1\), \(t_1 = [2,2,1]\) and \(g = 1\), then Riemann--Hurwitz formula implies that \(g_{\hat X} = 16\); however, according to the classification of group actions on surfaces of  low genus (up to \(48\)) in \cite{book:Breuer}*{chapter~5} implemented in GAP (see \cite{software:GAP}), there is no action of \(\Alt_5\) with signature \((1;2)\) on a Riemann surface of genus~\(16\). That proves \cref{item:GP221a}.

		Existence for \cref{item:GP221b} when \(g\geq 2\) is granted by the generating \((g;2)\)-vector
		\(
		((1\,3\,4\,2\,5),\allowbreak (1\,5\,2\,3\,4),\allowbreak (1\,3\,2\,4\,5),\allowbreak (1\,5\,4\,3\,2),\allowbreak \Id_{\times (2g - 4)}, (1\,2)(3\,4))
		\); if \(g = 1\) and \(n > 1\), then, according to \cref{lemm:AnyCycleByAProductInA5}, we can choose a permutation \(c_i\) in \(\Alt_5\) of type \(t_i\) for each \(i\in \{1,\ldots, n\}\) such that \(\prod_{i=1}^n c_i = (1\,3\,2)\); thereby, we get the generating \((1;m_1,\ldots,m_n)\)-vector \(
		((1\,4\,3\,2\,5), (1\,5\,2\,4\,3),c_1,\ldots,c_n)\).

		For existence in the case of \cref{item:GP221c}, by \cref{lemm:ProdsInD5Even221}, we can set \(c_i\) as a permutation of type \(t_i\) in \(\Dih_5\) for each \(i\in\{1,\ldots,n\}\) such that \(\prod_{i=1}^n c_i = (1\,5\,4\,3\,2)\). Set \(a_1\) and \(b_1\) according to \(G\) as given by \cref{table:ChoiceOfAiBiWith5Comm}, also set \(a_i \coloneqq \Id\) and \(b_i \coloneqq \Id\) for \(i\in\{2,\ldots,g\}\); so \(\prod_{i=1}^{g}[a_i,b_i] = [a_1,b_1] = (1\,2\,3\,4\,5)\) and \(\angp{a_1,b_1} = G\) in each case.
		\begin{table}
			\caption{Choice of permutations with commutator of type \([5]\) that generate \(\Alt_5\), \(\GA\) or \(\Dih_5\)}\label{table:ChoiceOfAiBiWith5Comm}
			\begin{tabular}{@{} *4{>{\(} l <{\)}} @{}}\toprule
				G	& a_1	& b_1	& [a_1,b_1]	 \\\midrule
				\Alt_5	& (1\,2\,5\,3\,4)	& (1\,5\,3\,2\,4) & (1\,2\,3\,4\,5)\\
				\GA		& (1\,3\,2\,5)		& (3\,2\,4\,5)	& (1\,2\,3\,4\,5)\\
				\Dih_5	& (1\,3\,5\,2\,4)	& (1\,2)(3\,5)	& (1\,2\,3\,4\,5)\\\bottomrule
			\end{tabular}
		\end{table}

		\textsc{\Cref{item:GPonly5}.} Suppose that \(f\) has only one branch value, of type \([5]\), and \(G = \Cyc_5\). \Cref{eq:genVector} yields that there is a permutation with cycle structure \([5]\) in \(\Cyc_5' = \{\Id\}\), a contradiction. That contradiction implies \cref{item:GPonly5a}.

		For existence in the case where \(G = \Cyc_5\), according to \cref{lemm:ProdsInD5Even221}, we can choose a permutation \(c_i\) of type \([5]\) in \(\Cyc_5\) for each \(i\in\{2,\ldots,n\}\) such that \(\prod_{i = 2}^n c_i = (1\,2\,3\,4\,5)\). Set \(c_1\coloneqq (1\,5\,4\,3\,2)\) and each \(a_i\) and \(b_i\) as \(\Id\); thereby, we get a generating \((g;5_{\times n})\)-vector of \(\Cyc_5\). When \(G\) is \(\Dih_5\), \(\GA\) or \(\Alt_5\), according to \cref{lemm:ProdsInD5Even221} again, we can choose a permutation \(c_i\) of type \([5]\) in \(\Cyc_5\) for each \(i\in\{2,\ldots,n\}\) such that \(\prod_{i = 2}^n c_i = (1\,5\,4\,3\,2)\). Set \(a_1\) and \(b_1\) according to \(G\) as given by \cref{table:ChoiceOfAiBiWith5Comm}, also set \(a_i \coloneqq \Id\) and \(b_i \coloneqq \Id\) for \(i\in\{2,\ldots,g\}\); so \(\prod_{i=1}^{g}[a_i,b_i] = (1\,2\,3\,4\,5)\) and \(\angp{a_1,b_1} = G\) in each case.
	\end{proof}

	Now we determine each possible \(G\) according to the ramification of \(f\) in the special case where \(Y \cong \PP^1\).

	\begin{theo}\label{theo:ClassificationG0}
		Consider \(5\)-fold covering \(f\colon X\to~\PP^1\) with ramification data \((t_1,\ldots, t_n)\). The following statements hold:
		\begin{enumerate}
			\item\label{item:G032or2111} If at least one \(t_i\) equals \([2,1,1,1]\) or \([3,2]\), then \(G = \Sym_5\).

			\item\label{item:G0311and41} If there are types \([3,1,1]\) and \([4,1]\) in \((t_1,\ldots,t_n)\), then \(G= \Sym_5\).

			\item\label{item:G0S5orGA} Suppose there are no types \([3,2]\), \([3,1,1]\) or \([2,1,1,1]\) in \((t_1,\ldots, t_n)\), but there is at least one type \([4,1]\). We have:
			\begin{enumerate}
				\item\label{item:G0S5orGAa} If \(\deg(t_1,\ldots,t_n) = 8\), then \(G \cong \GA\).
				\item\label{item:G0S5orGAb} If \(\deg(t_1,\ldots,t_n) > 8\), then \(G \cong \GA\) or \(G = \Sym_5\).
			\end{enumerate}

			\item\label{item:G0A5} If \(f\) has only even branch values and at least one of them is of type \([3,1,1]\), then \(G = \Alt_5\).

			\item\label{item:G0221} Suppose that there are neither odd nor \([3,1,1]\) types in \((t_1,\ldots, t_n)\), but there is at least one type \([2,2,1]\). We have three cases:
			\begin{enumerate}
				\item\label{item:G0221a} If \(\deg(t_1,\ldots,t_n) = 8\), then \(G\cong \Dih_5\).

				\item\label{item:G0221b} If there is an odd number of branch values of type \([2,2,1]\), then \(G = \Alt_5\).

				\item\label{item:G0221c} If \(\deg(t_1,\ldots,t_n) > 8\) and there is an even number of branch values of type \([2,2,1]\), then \(G \cong \Dih_5\) or \(G = \Alt_5\).
			\end{enumerate}

			\item\label{item:G0only5} Suppose that \(t_i = [5]\) for each \(i\in \{1,\ldots,n\}\). If \(n = 2\), then \(G \cong \Cyc_5\); otherwise, we have \(G \cong \Cyc_5\) or \(G = \Alt_5\).
		\end{enumerate}

		Conversely, for any realizable tuple \((t_1,\ldots,t_n)\) of ramification types and for each possible \(G\) listed in \cref{item:G032or2111,item:G0311and41,item:G0S5orGA,item:G0A5,item:G0221,item:G0only5}, there is a \(5\)-fold covering with that ramification data and that monodromy group.
	\end{theo}

	\begin{proof}
		We proceed in the same manner as in the proof of \cref{theo:ClassificationGP}.

		\textsc{\Cref{item:G032or2111}.} \Cref{table:CycleStructure} shows that the only transitive subgroup of \(\Sym_5\) that contains permutations of type \([3,2]\) or \([2,1,1,1]\) is \(\Sym_5\) itself; hence \(G = \Sym_5\).

		\textsc{\Cref{item:G0311and41}.} Same as \cref{item:G032or2111}.

		\textsc{\Cref{item:G0S5orGA}.} \Cref{table:CycleStructure} shows that the only transitive subgroups of \(\Sym_5\) that contain permutations of type \([4,1]\) are \(\GA\) and \(\Sym_5\); hence \(G \cong \GA\) or \(G = \Sym_5\), which proves \cref{item:G0S5orGAb}. Now suppose that \(\deg(t_1,\ldots,t_n) = 8\), since \((t_1,\ldots,t_n)\) is even, there must be at least two types \([4,1]\), say \(t_1\) and \(t_2\), so \(n = 3\) and \(\deg(t_3) = 2\); hence, Riemann--Hurwitz formula implies that \(g_{\hat X} = 1\). Since the full automorphism group of a surface of genus \(1\) is isomorphic to \(\ZZ^2 \rtimes \Cyc_k\), where \(k \in \{2,4,6\}\) (see \cite{book:Breuer}*{subsection~3.4}), which is a solvable group, the non-solvable group \(\Sym_5\) cannot act on \(\hat X\). That proves \cref{item:G0S5orGAa}.

		Since now there two possibilities for \(G\), we give separate existence proofs.
        
		\textsc{Case \(G = \Sym_5\).} Assume \(t_1 = t_2 = [4,1]\). Since \((t_3,\ldots, t_n)\) is also even and \(\deg(t_3,\ldots,t_n) \geq 4\), according to \cref{lemm:Type5CycleAsProduct}, we can choose a permutation \(c_i\) of type \(t_i\) for each \(i\in\{3,\ldots, n\}\) such that \(\prod_{i = 3}^n c_i = (1\,5\,4\,3\,2)\). Set \(c_1\coloneqq (1\,2\,5\,4)\) and \(c_2 \coloneqq (2\,5\,3\,4)\), then \(\prod_{i = 1}^n c_i = \Id\); also, since \(c_2^{-1}c_1 = (1\,5\,4)\) and the only transitive subgroup of \(\Sym_5\) that contains \([3,1,1]\) and \([4,1]\) elements is \(\Sym_5\) itself, we have \(\angp{c_1,\ldots,c_n} = \Sym_5\).

		\textsc{Case \(G \cong \GA\).} Set \(c_i\) as any permutation of type \(t_i\) in \(\GA\) for \(i \in\{3,\ldots, n\}\). The product \(\prod_{i = 3}^{n} c_i\) must be of type \([2,2,1]\) or \([5]\) because it is even and an element of \(\GA\); without loss of generality, assume that it equals \((1\,4)(2\,3)\) or \((1\,5\,4\,3\,2)\). Set
		\[
		(c_1,c_2) \coloneqq \begin{cases}
			((1\,3\,2\,5), (1\,4\,5\,2)) & \text{if \(\prod_{i = 3}^{n} c_i = (1\,4)(2\,3)\),}\\
			((2\,4\,5\,3),(1\,3\,2\,5)) & \text{if \(\prod_{i = 3}^{n} c_i = (1\,2\,3\,4\,5)\).}
		\end{cases}
		\]
		Since \(\angp{c_1,c_2}\) generates a transitive subgroup of \(\Sym_5\) contained in \(\GA\) and with type \([4,1]\) permutations, we have \(\angp{c_1,\ldots,c_n} = \GA\); so \((c_1,\ldots,c_n)\) is a suitable generating \((0;4,4,t_3,\ldots,t_n)\)-vector of \(\GA\).

		\textsc{\Cref{item:G0A5}.} According to \cref{table:CycleStructure}, the possible monodromy groups \(G\) that contain permutations of type \([3,1,1]\) are \(\Sym_5\) and \(\Alt_5\); but every even permutation in \(\Sym_5\) is contained in \(\Alt_5\), so there are not suitable generating vectors for \(\Sym_5\) in this case. Hence \(G = \Alt_5\).

		\textsc{\Cref{item:G0221}.} According to \cref{table:CycleStructure}, \(G\) could be conjugate to any transitive subgroup of \(\Sym_5\) but \(\Cyc_5\); however, neither \(\Sym_5\) nor \(\GA\) are generated by their even permutations because all of them are contained in \(\Alt_5\) and \(\Dih_5\), respectively. Hence \(G = \Alt_5\) or \(G \cong \Dih_5\).

		Now suppose \(\deg(t_1,\ldots,t_n) = 8\), then, maybe after a re-enumeration, there are two possibilities:
		\begin{itemize}
			\item \(n=4\) and \(t_1 = t_2 = t_3 = t_4 = [2,2,1]\); or
			\item \(n=3\), \(t_1 = t_2 = [2,2,1]\) and \(t_3 = [5]\).
		\end{itemize}
		So \(\hat f\) has signature \((0;2,2,2,2)\) or \((0;2,2,5)\). If \(G = \Alt_5\), Riemann--Hurwitz formula yields \(g_{\hat X} = 1\) for signature \((0;2,2,2,2)\) and \(g_{\hat X} = -7\) for signature \((0;2,2,5)\); the second case is clearly not possible. Moreover, recall from the proof of \cref{item:G0S5orGA} that the full automorphism group of a surface of genus \(1\) is a solvable group; therefore, the non-solvable group \(\Alt_5\) cannot be a subgroup of it. That proves \cref{item:G0221a}.

		\Cref{lemm:ProdsInD5} states that a product \(\prod_{i = 1}^n c_i\) of permutations in \(\Dih_5\) with an odd number of type \([2,2,1]\) factors is of type \([2,2,1]\); that proves \cref{item:G0221b}.

		If \(\deg(t_1,\ldots,t_n) > 8\) and there is an even number of types \([2,2,1]\) in \((t_1,\ldots,t_n)\), both \(\Alt_5\) and \(\Dih_5\) are possible monodromy groups, which proves \cref{item:G0221c}. Since in the latter case there are two possibilities for \(G\), we prove existence in both cases separately.

		\textsc{Case \(G = \Alt_5\).} If there is a type \([5]\) in \((t_1,\ldots,t_n)\), say \(t_1\), then, modulo re-enumeration of \((t_1,\ldots,t_n)\), we have \(t_2 = t_3 = [2,2,1]\) and \(\deg(t_4,\ldots,t_n) \geq 4\). According to \cref{lemm:Type5CycleAsProduct}, we can choose a permutation \(c_i\) in \(\Alt_5\) of type \(t_i\) for each \(i\in \{4,\ldots,n\}\) such that \(\prod_{i=4}^n c_i = (1\,5\,4\,3\,2)\); set \(c_1\coloneqq (1\,4\,5\,3\,2)\), \(c_2\coloneqq (1\,2)(3\,5)\) and \(c_3 \coloneqq (1\,5)(2\,3)\), so \(\prod_{i = 1}^n c_i = \Id\). Also, note that \(c_1c_2 = (2\,4\,5)\), hence \(\angp{c_1,\ldots,c_n}\) contains only even permutations and of every possible type, so \(\angp{c_1,\ldots, c_n} = \Alt_5\). Hence \((c_1,\ldots,c_n)\) is a suitable generating \((0;5,2,2,m_4,\ldots,m_n)\)-vector of \(\Alt_5\).

		If there are no types \([5]\) in \((t_1,\ldots,t_n)\), then \(t_i = [2,2,1]\) for each \(i\in\{1,\ldots,n\}\), and \(n \geq 6\). By \cref{lemm:Type5CycleAsProduct}, we can choose a permutation \(c_i\) of type \([2,2,1]\) for each \(i\in \{5,\ldots,n\}\) such that \(\prod_{i=5}^n c_i = (1\,5\,4\,3\,2)\); set \(c_1 \coloneqq (1\,5)(2\,3)\), \(c_2 \coloneqq (1\,4)(2\,5)\), \(c_3 \coloneqq (1\,2)(3\,5)\) and \(c_4 \coloneqq (1\,5)(2\,3)\), so \(\prod_{i = 1}^n = \Id\). We also have that \(c_1c_2c_3 = (2\,4\,5)\), hence \(\angp{c_1,\ldots,c_n} = \Alt_5\). Therefore \((c_1,\ldots,c_n)\) is a generating \((0;2,\ldots,2)\)-vector for \(\Alt_5\).

		\textsc{Case \(G \cong \Dih_5\).} There are at least two types \([2,2,1]\) in \((t_1,\ldots,t_n)\), say \(t_1\) and \(t_2\). By \cref{lemm:ProdsInD5Even221} we can choose a permutation \(c_i\) of type \(t_i\) in \(\Dih_5\) for each \(i\in \{3,\ldots,n\}\) such that \(\prod_{i=3}^{n} c_i = (1\,5\,4\,3\,2)\). Set \(c_1 \coloneqq (1\,3)(4\,5)\) and \(c_2 \coloneqq (1\,2)(3\,5)\) (both permutations belong to \(\Dih_5\)), so \(\prod_{i=1}^{n} c_i = \Id\) and \(\angp{c_1,\ldots,c_n} = \Dih_5\). Hence \((c_1,\ldots,c_n)\) is a generating \((0;2,2,m_3,\ldots,m_n)\)-vector of \(\Dih_5\).

		\textsc{\Cref{item:G0only5}.} Note that each type \([5]\) permutation in \(\Dih_5\) or \(\GA\) actually belongs to \(\Cyc_5\), so neither \(\Dih_5\) nor \(\GA\) can be generated just by elements of type \([5]\); also, each type \([5]\) permutation in \(\Sym_5\) is contained in \(\Alt_5\), so \(\Sym_5\) cannot be generated by elements of type \([5]\). Therefore, the monodromy group \(G\) must be conjugate to \(\Cyc_5\) or \(\Alt_5\). If \(n=2\), then \(G\) is generated by a single permutation of type \([5]\), hence \(G \cong \Cyc_5\). That proves \cref{item:G0only5}. ow we give an existence proof.

		\textsc{Case \(G \cong \Cyc_5\).} According to \cref{lemm:ProdsInD5Even221}, we can choose a permutation \(c_i\) of type \([5]\) in \(\Cyc_5\) for each \(i\in\{2,\ldots,n\}\) such that \(\prod_{i = 1}^n c_i = (1\,5\,4\,3\,2)\); set \(c_1 \coloneqq (1\,2\,3\,4\,5)\), so \(\prod_{i=1}^{n} c_i = \Id\) and \(\angp{c_1,\ldots, c_n} = \Cyc_5\). Therefore, \((c_1,\ldots,c_n)\) is a generating \((0;5,\ldots,5)\)-vector of \(\Cyc_5\).

		\textsc{Case \(G = \Alt_5\).} According to \cref{lemm:Type5CycleAsProduct}, we can choose a permutation \(c_i\) of type \([5]\) for each \(i\in\{3,\ldots,n\}\) such that \(\prod_{i = 3}^n c_i = (1\,5\,4\,3\,2)\). Set \(c_1 \coloneqq (1\,3\,2\,5\,4)\) and \(c_2 \coloneqq (1\,3\,5\,4\,2)\), so \(\prod_{i=1}^{n}c_i = \Id\). Note that \(c_1^2c_2^2 = (2\,5\,4)\), hence \(\angp{c_1,\ldots,c_n} = \Alt_5\). Hence \((c_1,\ldots,c_n)\) is a generating \((0;5,\ldots,5)\)-vector of \(\Alt_5\).
	\end{proof}

	\section{Decomposition of the Jacobian of the Galois closure}\label{sec:DecompositionOfJacobian}

	In this section we give, up to isogeny, the group algebra decomposition of \(\J(\hat X)\) in terms of the Jacobian and Prym varieties associated to the intermediate coverings of \(\hat f\) for each possible monodromy group \(G\) given in \cref{theo:ClassificationGP,theo:ClassificationG0}.

	\subsection{Cyclic monodromy group}

	Suppose that \(G \cong \Cyc_5\). There are just two rational irreducible representations of \(\Cyc_5\): one of degree~\(1\), the trivial representation, and one of degree~\(4\), the restriction of the \emph{standard representation} (see \cite{book:FultonHarris}*{p.~27}); both representations have Schur index~\(1\). 
%

	Moreover, since a complex irreducible representation Galois associated to \(V\) is of degree~\(1\) (see \cite{book:Serre}*{section~5.1}), \cref{theo:GroupAlgebraDecomposition} implies that the group algebra decomposition of \(\J(X)\) is of the form \(\J(X) \sim A_1 \times A_2\).

	Besides, according to \cref{theo:ClassificationG0,theo:ClassificationGP}, the map \(f\) is unramified or \(R_f = \sum_{j=1}^{n} 4p_j\) with \(n \geq 2\); the former case is only possible if \(g \geq 1\). Also, note that in this case \(f\) is already Galois because \(\ord{G} = \deg f\); since \(5\) is prime, the map \(f\) has not intermediate coverings and it is cyclic.

	\begin{theo}\label{theo:CaseC5}
		If \(G \cong \Cyc_5\), then \(\J(\hat X) \sim \J(Y) \times \Prym(f)\)
		is the group algebra decomposition of \(\J(\hat X)\). We have \(\dim \Prym(f) = 4g - 4 + 2n\); moreover, the polarization induced on \(\J(Y)\times \Prym(f)\) is of type \((1_{\times (3g - 3)}, 5_{\times (2g - 1)})\) if \(f\) is étale, and of type \((1_{\times (3g - 4 + 2n)}, 5_{\times 2g})\) otherwise.
	\end{theo}
	\begin{proof}
		Frobenius reciprocity implies \(\rho_{\{\Id\}} = U \oplus V\) and \(\rho_{\Cyc_5} = U\). Therefore, \cref{prop:GAComponentAsPrym} implies that \(\J(\hat X) \sim \J(Y) \times \Prym(f)\).

		From Riemann--Hurwitz formula, we get \(g_X = 5g - 4 + 2n\). Hence \(\dim \J(\hat X) = 5g - 4 + 2n\) and \(\dim \Prym(f) = 4g - 4 + 2n\).

		\Cref{prop:CyclicEtaleFactor} yields that
		\[
		\left|\ker f^* \right| = \begin{cases}
			5 &\mbox{if \(n = 0\)}, \\
			1 & \mbox{if \(n \geq 2\)}.
		\end{cases}
		\]
		By \cref{theo:PrymAndPolarizations}, the polarization \((f^*)^* \Theta_{f^* \J(Y)}\) is analytically equivalent to \(\Theta_Y^{\otimes 5}\), and hence type \((5_{\times g})\). Also,
		\[
		\ker(\Theta_{\Prym(f)}) = \ker(\Theta_{f^* \J(Y)}) \cong \dfrac{(\ker f^*)^\perp}{\ker f^*} \cong  \begin{cases}
			\Cyc_5^{2g - 2} &\mbox{if \(n = 0\)}, \\
			\Cyc_5^{2g} & \mbox{if \(n \geq 2\)};
		\end{cases}
		\]
		therefore,
		\[
		\type(\Theta_{\Prym(f)}) = \begin{cases}
			(1_{\times (3g - 3)}, 5_{\times (g - 1)}) & \text{if \(n = 0\),}\\
			(1_{\times (3g - 4 + 2n)}, 5_{\times g}) & \text{if \(n\geq 2\).}
		\end{cases}
		\]
		That implies the assertion on the type of the polarization.
	\end{proof}

	\subsection{Dihedral monodromy group}

	Suppose that \(G \cong \Dih_5\). We set \(\Cyc_5 \coloneqq \angp{(1\,2\,3\,4\,5)}\) and \(\Cyc_2 \coloneqq \angp{(2\,5)(3\,4)}\). The subgroup lattice of \(\Dih_5\) yields the following intermediate coverings of \(\hat f\):
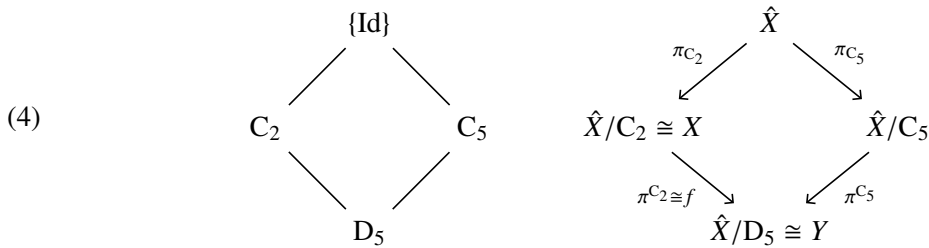
	\begin{equation}\label{eq:D5Diagram}
		\begin{tikzcd}
			&[-1em] \{\Id\} \ar[dl,dash] \ar[dr,dash] &[-1em] & &[-3em] \hat X \ar[dl,"\pi_{\Cyc_2}"'] \ar[dr, "\pi_{\Cyc_5}"] &[-2em] \\
			\Cyc_2 \ar[dr,dash] && \Cyc_5 \ar[dl,dash]  & \hat X/\Cyc_2 \cong X \ar[dr,"\pi^{\Cyc_2}\cong f"'] && \hat X/\Cyc_5 \ar[dl,"\pi^{\Cyc_5}"]\\
			& \Dih_5 &&& \hat X/\Dih_5 \cong Y &
		\end{tikzcd}
	\end{equation}
	Since \(\Stab_{\Dih_5}(1) = \Cyc_2\), we have \(\hat X/\Cyc_2 \cong X\) and \(\pi^{\Cyc_2} \cong f\).

	According to \cref{theo:ClassificationG0,theo:ClassificationGP}, we have
	\begin{equation}\label{eq:RamificationDivD5}
		R_f = \sum_{j=1}^{n_1} 4 p_j + \sum_{j=1}^{n_2} (q_j + r_j),
	\end{equation}
	where \(n_2\) is even; also:
	\begin{itemize}
		\item  if \(g = 1\), then \(n_1\) and \(n_2\) cannot be both zero; and
		\item  if \(g = 0\), then  \(n_2 \geq 2\) and if \(n_1 = 0\), then \(n_2 \geq 4\).
	\end{itemize}
	\Cref{eq:RamificationDivD5} implies that the signature of \(\hat f\) is \((g; 2_{\times n_2}, 5_{\times n_1})\). The genera and total ramification of intermediate coverings in \cref{eq:D5Diagram} were computed through the Sage implementation of \cref{subsec:GaloisClosure} and are presented in \cref{table:TotalRamificationsD5}.
	\begin{table}
		\caption{Intermediate coverings of \(\hat f\) where \(G \cong \Dih_5\)}
		\label{table:TotalRamificationsD5}
		\begin{tabular}{@{} >{\(}l<{\)} *{3}{>{\(}r<{\)}} @{}}\toprule
			\multicolumn{1}{@{}l}{\(H\)} & \multicolumn{1}{l}{Genus of \(\hat X/H\)}& \multicolumn{1}{l}{\(\deg(R_{\pi_H})\)} & \multicolumn{1}{l}{\(\deg(R_{\pi^H})\)} \\
			\midrule
			\{\Id\} &  10g_Y + 4n_1 + 5n_2/2 - 9 & 0 			& 8 n_1 + 5n_2 \\
			\Cyc_2 & 	 5g_Y+2n_1+n_2-4 		& n_2 			& 4 n_1 + 2n_2 \\
			\Cyc_5 & 2g_Y+ n_2 /2-1				& 8n_1 			& n_2 \\
			\Dih_5 &  g_Y 						& 8n_1 + 5n_2 	& 0 \\
			\bottomrule
		\end{tabular}
	\end{table}

	According to \cite{book:Serre}*{section~5.3}, there are four complex irreducible representations of \(\Dih_5\): two of degree~\(1\), the trivial one denoted by \(U\) and a non-trivial one denoted by \(W\), and two of degree~\(2\), which are not rational, but its direct sum is; furthermore, it is the restriction of the standard representation, so we will denote it by \(V\). So the rational conjugacy classes of \(\Dih_5\) are also three: the class of the identity, the class of \((2\,5)(3\,4)\) and the class of \((1\,2\,3\,4\,5)\).
	Rational irreducible representations of \(\Dih_5\) satisfy the following properties:
	\begin{itemize}
		\item \(m_U = m_{W} = m_{V} = 1\);
		\item \(U\) and \(W\) are complex irreducible representations; and
		\item \(V\) is Galois associated to a degree~\(2\) complex irreducible representation.
	\end{itemize}
	Hence, by \cref{theo:GroupAlgebraDecomposition}, the group algebra decomposition of \(\J(\hat X)\) is of the form
	\begin{equation}\label{eq:D5DecompositionForm}
		\J(\hat X) \sim A_1 \times A_2 \times A_3^2.
	\end{equation}

	\begin{theo}\label{theo:CaseD5}
        If \(G \cong \Dih_5\), then
		\begin{equation}\label{eq:D5Decomposition}
			\J(\hat X) \sim \J(Y) \times \Prym\paren{\pi^{\Cyc_5}} \times  \Prym(f)^2
		\end{equation}
		is the group algebra decomposition of \(\J(\hat X)\), where \(\Dih_5\) acts trivially on \(\J(Y)\), and as multiples of \(W\) and \(V\) on \(\Prym\paren{\pi^{\Cyc_5}}\) and \(\Prym(f)^2\), respectively. The dimensions of the Prym varieties involved are:
		\begin{itemize}
			\item \(\dim \Prym\paren{\pi^{\Cyc_5}}  = g + n_2/2 - 1\)
			\item \(\dim \Prym(f) = 4g + 2n_1 + n_2 - 4\)
		\end{itemize}
		and the types of their polarization types are:
		\begin{itemize}
			\item \(\type \Theta_{\Prym(\pi^{\Cyc_5})} = \begin{cases}
				(2_{\times (g - 1)})	&\text{if \(n_2 = 0\),}\\
				(1_{\times (n_2/2 - 1)}, 2_{\times g})	&\text{if \(n_2 \geq 2\);}
			\end{cases}\)

			\vspace{1ex}

			\item \(\type \Theta_{\Prym(f)} = (1_{\times (3g + 2n_1 + n_2 - 4)}, 5_{\times g})\).
		\end{itemize}
	\end{theo}
	\begin{proof}
        Frobenius reciprocity implies that \(\rho_{\Cyc_5} = U \oplus W\), \(\rho_{\Cyc_2} = U \oplus V\) and \(\rho_{\Dih_5} = U\). Therefore, in the notation of \cref{eq:D5DecompositionForm}, \cref{prop:GAComponentAsPrym} implies that \(A_1 \sim \J(Y)\), \(A_2 \sim \Prym\paren{\pi^{\Cyc_5}}\) and \(A_3 \sim \Prym(f)\). The dimension of those Jacobian and Prym varieties are directly computed from the genera of the corresponding curves in \cref{table:TotalRamificationsD5}.

		According to \cref{table:TotalRamificationsD5}, we have that \(\pi^{\Cyc_5}\) is étale if and only if \(n_2 = 0\); besides, the map \(\pi^{\Cyc_5}\) is cyclic (it is of degree~\(2\)) whereas \(f\) is not even Galois, because \(\Cyc_2\) is not normal in \(\Dih_5\). Also, \cref{prop:IntermediateOfGalois} implies that neither \(\pi^{\Cyc_5}\) nor \(f\) factor non-trivially. Hence, by \cref{prop:CyclicEtaleFactor},
		\begin{equation*}
			\ord{\ker \pi^{\Cyc_5*}} = \begin{cases}
				2& \text{if \(n_2 = 0\),}\\
				1& \text{if \(n_2 \geq 2\),}
			\end{cases}
		\end{equation*}
		and \(\ord{\ker f^*} = 1\).	From \cref{item:DecompositionInJacAndPrym} of \cref{theo:PrymAndPolarizations} and \cref{table:TotalRamificationsD5}, we get the desired polarization types.
	\end{proof}


	\subsection{Affine monodromy group}

	Suppose that \(G \cong \GA\) as defined in \cref{prop:PossibleMonodromy}. Since \(\angp{(1\,2\,3\,4\,5)}\) and \(\angp{(2\,3\,5\,4)}\) are cyclic groups of order \(5\) and \(4\), respectively, and \(\angp{(2\,5)(3\,4),(1\,2\,3\,4\,5)}\) is dihedral of order \(10\); we will denote these three groups by \(\Cyc_5\), \(\Cyc_4\) and \(\Dih_5\), respectively. The subgroup lattice of \(\GA\) yields the following commutative diagram:
	\begin{equation}\label{eq:GADiagram}
		\begin{tikzcd}
			\{\Id\} \ar[d,dash] \ar[dr,dash] && \hat X \ar[d,"\pi_{\Cyc_4}"'] \ar[dr,"\pi_{\Cyc_5}"] &\\
			\Cyc_4 \ar[dd,dash] &\Cyc_5 \ar[d,dash] & \hat X/\Cyc_4 \cong X \ar[dd,"\pi^{\Cyc_4}\cong f"'] & \hat X/\Cyc_5 \ar[d,"\pi^{\Cyc_5}_{\Dih_5}"]\\
			& \Dih_5 \ar[dl,dash] && \hat X/\Dih_5 \ar[dl,"\pi^{\Dih_5}"]\\
			\GA& & \hat X/\GA \cong Y&
		\end{tikzcd}
	\end{equation}
	Since \(\Stab_{\GA}(1) = \Cyc_4\) we have \(\hat X/\Cyc_4 \cong X\) and \(\pi^{\Cyc_4} \cong f\).

	According to \cref{theo:ClassificationG0,theo:ClassificationGP}, we have
	\begin{equation}\label{eq:RamificationDivGA}
		R_f = \sum_{j=1}^{n_1} 4 p_j + \sum_{j=1}^{n_2} (q_j + r_j) + \sum_{j=1}^{n_3}3s_j,
	\end{equation}
	where \(n_3\) is even; also:
	\begin{itemize}
		\item if \(g = 1\), then \(n_1\), \(n_2\) and \(n_3\) cannot be all zero; and
		\item if \(g = 0\), then \(n_3 \geq 2\) and if \(n_1 = n_2 = 0\), then \(n_3 \geq 4\).
	\end{itemize}
	\Cref{eq:RamificationDivGA} implies that the signature of \(\hat f\) is \((g; 2_{\times n_2},4_{ \times n_3}, 5_{\times n_1})\).
	The genera and total ramification of the several coverings in \cref{eq:GADiagram} were computed through the Sage implementation of \cref{subsec:GaloisClosure} and are presented in \cref{table:TotalRamificationsGA}.
	\begin{table}
		\caption{Intermediate coverings of \(\hat f\) where \(G \cong \GA\)}
		\label{table:TotalRamificationsGA}
		\begin{tabular}{@{} >{\(}l <{\)} *{3}{>{\(}r <{\)}} @{}}\toprule
			\multicolumn{1}{@{}l}{\(H\)} & \multicolumn{1}{l}{Genus of \(\hat X/H\)}& \multicolumn{1}{l}{\(\deg(R_{\pi_H})\)} & \multicolumn{1}{l}{\(\deg(R_{\pi^H})\)}\\
			\midrule
			\{\Id\} 	& 20g_Y + 8n_1+5n_2& 0 & 16n_1+10n_2+15n_3\\
			& +15n_3/2 -19 & & \\[1ex]
			\Cyc_4 	&5g_Y + 2n_1+n_2& 2n_2+3n_3 &4n_1+2n_2+3n_3 \\
			&+3n_3 /2-4	&&\\[1ex]
			\Cyc_5 	&4g_Y + n_2	& 16n_1 & 2n_2+3n_3 \\
			& + 3n_3/2 -3 &&\\[1ex]
			\Dih_5 	&2 g_Y + n_3 /2 -1 			& 16n_1+10n_2+5n_3 	&n_3 \\[1ex]
			\GA&g_Y						& 16n_1+10n_2+15n_3	&0\\
			\bottomrule
		\end{tabular}
	\end{table}
	Also, we have
	\begin{equation}\label{eq:IntermediateRamificationGA}
		\deg R_{\pi^{\Cyc_5}_{\Dih_5}} = 2n_2 + n_3.
	\end{equation}

	There are five complex irreducible representations of \(\GA\):
	\begin{itemize}
		\item Four of degree~\(1\): the trivial representation \(U\), the restriction of the alternating representation \(\tilde U\), and two representations, dual to each other, denoted by \(W\) and \(W^*\).

		\item One of degree \(4\): the restriction of the standard representation \(V\).
	\end{itemize}
%
%
	The representations \(W\) and \(W^*\) are not rational, but their direct sum \(W \oplus W^*\) is. Hence \(\Irr_\QQ(\GA) = \{U, \tilde U, W \oplus W^*, V\}\).
%
	\Cref{table:RationalTableGA} shows he rational character table of \(\GA\).

	\begin{table}
		\caption{Rational character table of \(\GA\)}
		\label{table:RationalTableGA}
		\begin{tabular}{@{} >{\(} l <{\)} *{4}{ >{\(} r <{\)} } @{}}\toprule
			&1	&4				&10				&5	\\
			\GA&()	&(1\,2\,3\,4\,5)	&(2\,3\,5\,4)	&(1\,4)(2\,3)\\
			\midrule
			U			&1	&1	&1	&1\\
			\tilde U	&1	&1	&-1	&1\\
			W\oplus W^*	&2	&2	&0	&-2\\
			V			&4	&-1	&0	&0\\
			\bottomrule
		\end{tabular}
	\end{table}

	Rational irreducible representations of \(\GA\) satisfies the following properties:
	\begin{itemize}
		\item \(m_U = m_{\tilde U} = m_{W \oplus W^*} = m_{V} = 1\);
		\item \(U\), \(\tilde U\) and \(V\) are complex irreducible representations; and
		\item \(W \oplus W^*\) is Galois associated to \(W\), which is of degree \(1\).
	\end{itemize}
	So, by \cref{theo:GroupAlgebraDecomposition}, the group algebra decomposition of \(\J(\hat X)\) is of the form
	\begin{equation}\label{eq:GADecompositionForm}
		\J(\hat X) \sim A_1 \times A_2\times A_3 \times A_4^4.
	\end{equation}

	\begin{theo}\label{theo:CaseGA}
		If \(G \cong \GA\), then
		\begin{equation}\label{eq:GADecomposition}
			\J(\hat X) \sim \J(Y) \times \Prym\paren{\pi^{\Dih_5}} \times  \Prym\paren{\pi_{\Dih_5}^{\Cyc_5}} \times \Prym(f)^4
		\end{equation}
		is the group algebra decomposition of \(\J(\hat X)\), where \(\GA\) acts trivially on \(\J(Y)\) and as multiples of \(\tilde U\), \(W\oplus W^*\) and \(V\) on \(\Prym\paren{\pi^{\Dih_5}}\),  \(\Prym\paren{\pi_{\Dih_5}^{\Cyc_5}}\) and \(\Prym(f)^4\), respectively. The dimensions of the abelian varieties involved are:
		\begin{itemize}
			\item \(\dim \Prym\paren{\pi^{\Dih_5}} = g + n_3 /2 - 1\)

			\item \(\dim \Prym\paren{\pi_{\Dih_5}^{\Cyc_5}} = 2g + n_2 + n_3 - 2\)

			\item \(\dim \Prym(f) = 4g + 2n_1 + n_2 + 3 n_3/2 - 4\)
		\end{itemize}
		and the types of their polarizations:
		\begin{itemize}
			\item \(\type \Theta_{\Prym(\pi^{\Dih_5})} = \begin{cases}
				(2_{\times (g - 1)})	&\text{if \(n_3 = 0\),}\\
				(1_{\times (n_3/2 - 1)}, 2_{\times g})	&\text{if \(n_3 \geq 2\);}
			\end{cases}\)

			\vspace{1ex}

			\item \(\type \Theta_{\Prym(\pi_{\Dih_5}^{\Cyc_5})} = \begin{cases}
				(2_{\times (2g - 2)})	&\text{if \(n_2 = n_3 = 0\),}\\
				(1_{\times (n_2 + n_3/2 - 1)}, 2_{\times g_{2 g + n_3 /2 -1 }})	&\text{if \(n_2 + n_3 > 0\);}
			\end{cases}\)

			\vspace{1ex}

			\item \(\type \Theta_{\Prym(f)} = (1_{\times (3g + 2n_1 + n_2+ 3n_3/2 - 4)}, 5_{g})\).
		\end{itemize}
	\end{theo}
	\begin{proof}
		Frobenius reciprocity and \cref{table:RationalTableGA} imply that \(\rho_{\Dih_5} = U \oplus \tilde U\),	\(\rho_{\Cyc_5} = U \oplus \tilde U \oplus (W \oplus W^*)\) and \(\rho_{\Cyc_4} = U \oplus V\).
		Therefore, following the notation of \cref{eq:GADecompositionForm}, \cref{prop:GAComponentAsPrym} implies that \(A_1 \sim \J(Y)\), \(A_2 \sim \Prym\paren{\pi^{\Dih_5}}\), \(A_3 \sim \Prym\paren{\pi_{\Dih_5}^{\Cyc_5}}\) and \(A_4 \sim \Prym(f)\).

		The dimension of the several Jacobian and Prym varieties are directly computed from the genera of the corresponding curves in \cref{table:TotalRamificationsGA}.

		According to \cref{table:TotalRamificationsGA} and \cref{eq:IntermediateRamificationGA}, we have that \(\pi^{\Dih_5}\) and \(\pi^{\Cyc_5}_{\Dih_5}\) are étale if and only if \(n_3 = 0\) and \(n_2 = n_3 = 0\), respectively; besides, the maps \(\pi^{\Dih_5}\) and \(\pi^{\Cyc_5}_{\Dih_5}\) are cyclic (both of them are of degree~\(2\)) whereas \(f\) is not even Galois (because \(\Cyc_4\) is not normal in \(\GA\)). Also, by \cref{prop:IntermediateOfGalois}, neither \(\pi^{\Dih_5}\) nor \(\pi^{\Cyc_5}_{\Dih_5}\) nor \(f\) factor non-trivially. Hence, by \cref{prop:CyclicEtaleFactor} we have that
		\begin{align*}
			\ord{\ker \pi^{\Dih_5*}} &= \begin{cases}
				2& \text{if \(n_3 = 0\),}\\
				1& \text{if \(n_3 \geq 2\);}
			\end{cases}\\
			\ord{\ker \pi_{\Dih_5}^{\Cyc_5*}} &= \begin{cases}
				2& \text{if \(n_2 = n_3 = 0\),}\\
				1& \text{if \(n_2 + n_3 > 0\),}
			\end{cases}
		\end{align*}
		and \(\ord{\ker f^*} = 1\). Using \cref{item:DecompositionInJacAndPrym} of \cref{theo:PrymAndPolarizations} we get the types of the polarizations.
	\end{proof}


	\subsection{Alternating monodromy group}

	Suppose that \(G = \Alt_5\) and set \(\Cyc_5 \coloneqq \angp{(1\,2\,3\,4\,5)}\), \(\Dih_5 \coloneqq \angp{(1\,2\,3\,4\,5), (2\,5)(3\,4)}\) and \(\Alt_4 \coloneqq \angp{(2\,3\,4), (3\,4\,5)}\). The subgroup lattice of \(\Alt_5\) yields the following commutative diagram:
	\begin{equation}\label{eq:A5Diagram}
		\begin{tikzcd}
			\{\Id\} \ar[dd,dash] \ar[dr,dash] &&[2em]\hat X \ar[dd,"\pi_{\Alt_4}"'] \ar[dr,"\pi_{\Cyc_5}"] &\\
			&\Cyc_5 \ar[d,dash]  && \hat X/\Cyc_5 \ar[d,"\pi^{\Cyc_5}_{\Dih_5}"]\\
			\Alt_4 \ar[d,dash] & \Dih_5 \ar[dl,dash]& \hat X/\Alt_4 \cong X \ar[d,"\pi^{\Alt_4}\cong f"'] & \hat X/\Dih_5  \ar[dl,"\pi^{\Dih_5}"]\\
			\Alt_5&& \hat X/\Alt_5 \cong Y&
		\end{tikzcd}
	\end{equation}
	Since \(\Stab_{\Alt_5}(1) = \Alt_4\), we have \(\hat X/\Alt_4 \cong X\) and \(\pi^{\Alt_4} \cong f\).

	According to \cref{theo:ClassificationG0,theo:ClassificationGP}, we have
	\begin{equation}\label{eq:RamificationDivA5}
		R_f = \sum_{j=1}^{n_1} 4 p_j + \sum_{j=1}^{n_2} (q_j + r_j) + \sum_{j=1}^{n_4}2t_j;
	\end{equation}
	also:
	\begin{itemize}
		\item if \(g = 1\), then \(n_1+n_2+n_4 > 0\), and if \(n_1 = n_4 = 0\), then \(n_2 \geq 2\); and
		\item if \(g = 0\), then \(\deg(R_f) \geq 8\) and if \(n_4 = 0\), then \(\deg(R_f) > 8\).
	\end{itemize}
    \Cref{eq:RamificationDivA5} implies that the signature of \(\hat f\) is \((g; 2_{\times n_2}, 3_{\times n_4}, 5_{\times n_1})\).
	The genera and total ramification of the several coverings in \cref{eq:A5Diagram} were computed through the Sage implementation of \cref{subsec:GaloisClosure} and are presented in \cref{table:TotalRamificationsA5}.
	\begin{table}
		\caption{Total ramification of the intermediate coverings of \(\hat f\) with \(G \cong \Alt_5\)}
		\label{table:TotalRamificationsA5}
		\begin{tabular}{@{} >{\(}l <{\)} *{3}{>{\(}r <{\)}} @{}}\toprule
			\multicolumn{1}{@{}l}{\(H\)} & \multicolumn{1}{l}{Genus of \(\hat X/H\)}& \multicolumn{1}{l}{\(\deg(R_{\pi_H})\)} & \multicolumn{1}{l}{\(\deg(R_{\pi^H})\)}\\
			\midrule
			\{\Id\} 	&60g_Y+24n_1+15n_2& 0 & 48n_1+30n_2+40n_4\\
			&+20n_4-59&&\\[1ex]
			\Cyc_5 	&12g_Y+4n_1+3n_2& 8n_1 & 8n_1+6n_2+8n_4 \\
			&+4n_4-11&&\\[1ex]
			\Dih_5 	&6g_Y+2n_1+n_2& 8n_1+10n_2 & 4n_1+2n_2+4n_4 \\
			&+2n_4-5 &&\\[1ex]
			\Alt_4 	&5g_Y+2n_1+n_2& 6n_2+16n_4& 4n_1+2n_2+2n_4 \\
			&+n_4-4&&\\[1ex]
			\Alt_5&g_Y	& 48n_1+30n_2+40n_4	& 0\\
			\bottomrule
		\end{tabular}
	\end{table}
	Also, we get that
	\begin{equation}\label{eq:IntermediateRamificationA5}
		\deg R_{\pi^{\Cyc_5}_{\Dih_5}} = 2n_2.
	\end{equation}

	According to \cite{book:FultonHarris}*{section~3.1}, \(\Alt_5\) has five complex irreducible representations:
	\begin{enumerate}
		\item The trivial representation, which we will denote by \(U\).

		\item One of degree \(4\); the restriction of the standard representation \(V\).

		\item One of degree \(5\), which we will denote by \(W\).

		\item Two of degree \(3\), denoted by \(W_2\) and \(W_3\), that satisfy \(W_2 \oplus W_3 = \Altp^2 V\).
	\end{enumerate}
%
%
	The representations \(W_2\) and \(W_3\) are not rational, but their direct sum \(\Altp^2 V\) is. So, the rational conjugacy classes of \(\Alt_5\) are also four: the classes of \(\Id\),  \((1\,2\,3)\), \((2\,5)(3\,4)\) and \((1\,2\,3\,4\,5)\).
	Hence \(\Irr_\QQ(\Alt_5) = \{U, V, W, \Altp^2 V\}\). The rational character table of \(\Alt_5\) is shown in \Cref{table:RationalTableA5}.
	\begin{table}
		\caption{Rational character table of \(\Alt_5\)}
		\label{table:RationalTableA5}
		\begin{tabular}{@{} >{\(} l <{\)} *{4}{ >{\(} r <{\)} } @{}}\toprule
			&1	&20			&15				&24	\\
			\Alt_5	&\Id	&(1\,2\,3)	&(1\,2)(3\,4)	&(1\,2\,3\,4\,5)\\
			\midrule
			U		&1	&1	&1	&1	\\
			V		&4	&1	&0	&-1	\\
			W		&5	&-1	&1	&0	\\
			\Altp^2V&6	&0	&-2	&1	\\
			\bottomrule
		\end{tabular}
	\end{table}
	Moreover, the rational irreducible representations of \(\Alt_5\) satisfies the following properties:
	\begin{itemize}
		\item \(m_U = m_{V} = m_{W} = m_{\Altp^2 V} = 1\);
		\item \(U\), \(V\) and \(W\) are complex irreducible representations; and
		\item \(\Altp^2 V\) is Galois associated to \(W_2\), which is of degree \(3\).
	\end{itemize}
	Hence, by \cref{theo:GroupAlgebraDecomposition}, the group algebra decomposition of \(\J(\hat X)\) is of the form
	\begin{equation}\label{eq:A5DecompositionForm}
		\J(\hat X) \sim A_1 \times A_2^4 \times A_3^5 \times A_4^3.
	\end{equation}

	\begin{theo}\label{theo:CaseA5}
		If \(G = \Alt_5\), then
		\begin{equation}\label{eq:A5Decomposition}
			\J(\hat X) \sim \J(Y) \times \Prym(f)^4 \times \Prym\paren{\pi^{\Dih_5}}^5   \times \Prym\paren{\pi_{\Dih_5}^{\Cyc_5}}^3
		\end{equation}
		is the group algebra decomposition of \(\J(\hat X)\), where \(\Alt_5\) acts trivially on \(\J(Y)\) and as multiples of \(V\), \(W\) and \(\Altp^2 V\) on \(\Prym(f)^4\), \(\Prym\paren{\pi^{\Dih_5}}^5\) and  \(\Prym\paren{\pi_{\Dih_5}^{\Cyc_5}}^3\), respectively. The dimensions of the Prym varieties involved are:
		\begin{itemize}
			\item \(\dim \Prym(f) = 4g + 2n_1 + n_2 + n_4 - 4\)

			\item \(\dim \Prym\paren{\pi^{\Dih_5}} = 5g + 2n_1 +n_2 + 2n_4 - 5\)

			\item \(\dim \Prym\paren{\pi_{\Dih_5}^{\Cyc_5}} = 6g + 2n_1 + 2n_2 + 2n_4 - 6\)
		\end{itemize}
		and the types of their polarizations are:
		\begin{itemize}
			\item \(\type \Theta_{\Prym(f)} = (1_{\times (3g + 2n_1 + n_2+ n_4 - 4)}, 5_{\times g})\);
            
            \vspace{1ex}

			\item \(\type \Theta_{\Prym(\pi^{\Dih_5})} = (1_{\times (4g + 2n_1 +n_2+ 2n_4 - 5)}, 6_{\times g})\);

			\vspace{1ex}

			\item \(\type \Theta_{\Prym(\pi_{\Dih_5}^{\Cyc_5})} = \begin{cases}
				(2_{\times (6g+2n_1+2n_4-6)})	&\text{if \(n_2 = n_3 = 0\),}\\
				(1_{\times (n_2 - 1)}, 2_{6g+2n_1+n_2 +2n_4-5})	&\text{if \(n_2 + n_3 > 0\).}
			\end{cases}\)
		\end{itemize}
	\end{theo}
	\begin{proof}
		Frobenius reciprocity and \cref{table:RationalTableA5} show that \(\rho_{\Alt_4} = U\oplus V\), \(\rho_{\Dih_5} = U \oplus W\) and \(\rho_{\Cyc_5} = U \oplus W \oplus \Altp^2 V\).
		Therefore, in the notation of \cref{eq:A5DecompositionForm}, \cref{prop:GAComponentAsPrym} implies that \(A_1 \sim \J(Y)\), \(A_2 \sim \Prym(f)\), \(A_3 \sim \Prym\paren{\pi^{\Dih_5}}\) and \(A_4 \sim \Prym\paren{\pi_{\Dih_5}^{\Cyc_5}}\). The dimension of the several Jacobian and Prym varieties are directly computed from the genera of the corresponding curves in \cref{table:TotalRamificationsA5}.

		According to \cref{eq:IntermediateRamificationA5}, we have that \(\pi^{\Cyc_5}_{\Dih_5}\) is étale if and only if \(n_2 = 0\); besides, the map \(\pi^{\Cyc_5}_{\Dih_5}\) is cyclic because it is of degree~\(2\), whereas neither \(f\) nor \(\pi^{\Dih_5}\) are Galois because \(\Alt_5\) is simple; moreover, \cref{prop:IntermediateOfGalois} implies that neither of those maps factor through a Galois covering onto \(Y\). Hence, by \cref{prop:CyclicEtaleFactor}, we have \(\ord{\ker f^*} = 1\), \(\ord{\ker \pi^{\Dih_5*}} = 1\) and
		\begin{equation*}
			\ord{\ker \pi_{\Dih_5}^{\Cyc_5*}} = \begin{cases}
				2& \text{if \(n_2 = 0\),}\\
				1& \text{if \(n_2 > 0\).}
			\end{cases}
		\end{equation*}
		From \cref{item:DecompositionInJacAndPrym} of \cref{theo:PrymAndPolarizations} and \cref{table:TotalRamificationsA5} we get the desired polarization types.
	\end{proof}


	\subsection{Symmetric monodromy group}

	Suppose that \(G = \Sym_5\) and set:
		\begin{itemize}
			\item \(\GA \coloneqq \angp{(1\,2\,3\,4\,5),(1\,2\,4\,3)}\)
			\item \(\Alt_5 \coloneqq \angp{(1\,2\,3\,4\,5), (1\,2\,3)}\)
			\item \(\Sym_4 \coloneqq \angp{(2\,3\,4\,5), (2\,3)}\)
			\item \(\Dih_6 \coloneqq \angp{(1\,2),(3\,4),(3\,4\,5)}\)
			\item \(\Dih_5 \coloneqq \angp{(1\,2\,3\,4\,5), (2\,5)(3\,4)}\)
			\item \(\Alt_4 \coloneqq \angp{(2\,3\,4), (3\,4\,5)}\)
			\item \(\Sym_3 \coloneqq \angp{(3\,4\,5),(3\,4)}\)
		\end{itemize}
	These subgroups induce the following intermediate coverings of \(\hat f\):
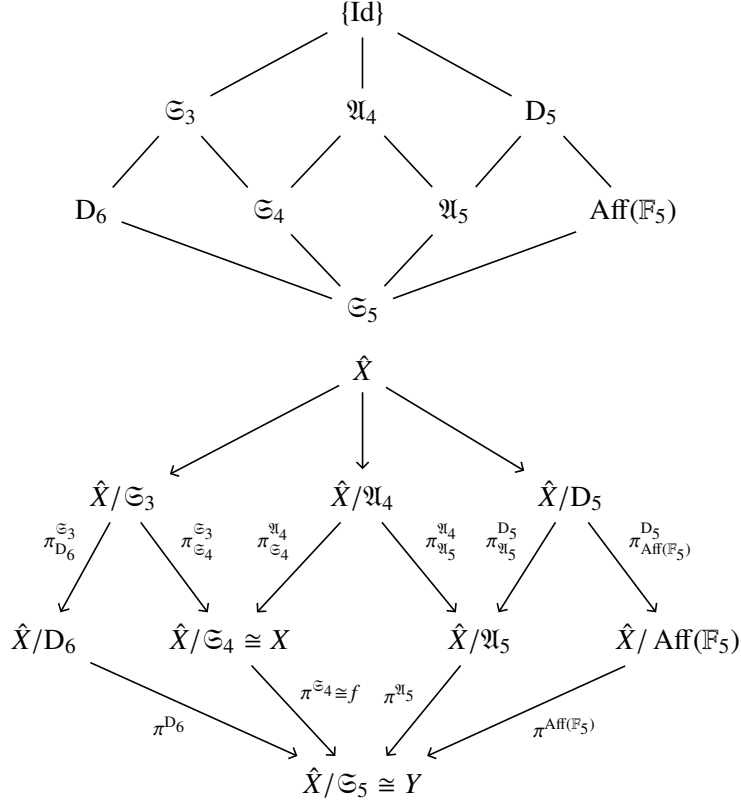
	\begin{gather*}\label{eq:S5Diagram}
			\begin{tikzcd}[column sep=1em,ampersand replacement=\&]
				\&\&\&\{\Id\} \ar[dll,dash] \ar[d,dash] \ar[drr,dash]\&\&\&[-0.8em]  \\
				\&\Sym_3 \ar[dl,dash] \ar[dr,dash]\&\&\Alt_4 \ar[dl,dash] \ar[dr,dash] \&\& \Dih_5 \ar[dl,dash] \ar[dr,dash] \&  \\
				\Dih_6 \ar[drrr,dash] \&\& \Sym_4 \ar[dr,dash] \&\& \Alt_5 \ar[dl,dash] \&\& \GA \ar[dlll,dash]\\
				\&\&\&\Sym_5\&\&\&
			\end{tikzcd}
        \\    
			\begin{tikzcd}[column sep=0em,ampersand replacement=\&]
				\&[-1ex]\&[-1ex]\&[-1ex]\hat X \ar[dll] \ar[d] \ar[drr]\&\&\&[-1ex]  \\[2ex]
				\&\hat X/\Sym_3 \ar[dl,"\pi_{\Dih_6}^{\Sym_3}"'] \ar[dr,"\pi_{\Sym_4}^{\Sym_3}"]\&\&\hat X/\Alt_4 \ar[dl,"\pi_{\Sym_4}^{\Alt_4}"'] \ar[dr,"\pi_{\Alt_5}^{\Alt_4}"] \&\& \hat X/\Dih_5 \ar[dl,"\pi_{\Alt_5}^{\Dih_5}"'] \ar[dr,"\pi_{\GA}^{\Dih_5}"] \&  \\[3ex]
				\hat X/ \Dih_6 \ar[drrr,"\pi^{\Dih_6}"'] \&\& \hat X/\Sym_4 \cong X \ar[dr,"\pi^{\Sym_4}\cong f"] \&\& \hat X/\Alt_5 \ar[dl,"\pi^{\Alt_5}"'] \&\& \hat X/\GA \ar[dlll,"\pi^{\GA}"]\\[3ex]
				\&\&\&\hat X/\Sym_5\cong Y \&\&\&
			\end{tikzcd}
	\end{gather*}
	Since \(\Stab_{\Sym_5}(1) = \Sym_4\), we have \(\hat X/\Sym_4 \cong X\) and \(\pi^{\Sym_4} \cong f\).

	According to \cref{theo:ClassificationG0,theo:ClassificationGP}, we have
	\begin{equation}\label{eq:RamificationDivS5}
		R_f = \sum_{j=1}^{n_1} 4 p_j + \sum_{j=1}^{n_2} (q_j + r_j) + \sum_{j=1}^{n_3} 3s_j + \sum_{j=1}^{n_4} 2 t_j + \sum_{j=1}^{n_5} (2u_j + v_j) + \sum_{j=1}^{n_6} w_j,
	\end{equation}
	where \(n_3\), \(n_5\) and \(n_6\) are even; also:
	\begin{itemize}
		\item If \(g=1\), then \(n_i\) cannot be zero for all \(i\in \{1,\ldots,6\}\).
		\item If \(g = 0\), then \(\deg(R_f) \geq 8\) and if \(n_4 = n_5 = n_6 = 0\), then \(\deg(R_f) > 8\).
	\end{itemize}

	\Cref{eq:RamificationDivS5} implies that the signature of \(\hat f\) is \((g; 2_{\times (n_2 + n_6)}, 3_{\times n_4}, 4_{\times n_3}, 5_{\times n_1}, 6_{\times n_5})\).
	The genera and total ramification of the several coverings in \cref{eq:S5Diagram} were computed through the Sage implementation of \cref{subsec:GaloisClosure} and are presented in \cref{table:TotalRamificationsS5}.
	\begin{sidewaystable}
		\centering
		\caption{Total ramification of the intermediate coverings of \(\hat f\) with \(G \cong \Sym_5\)}
		\label{table:TotalRamificationsS5}
		\begin{tabular}{@{} >{\(}l <{\)} *{3}{>{\raggedleft\arraybackslash\(} p{5.7cm}  <{\)}} @{}}\toprule
			\multicolumn{1}{@{}l}{\(H\)} & \multicolumn{1}{l}{Genus of \(\hat X/H\)}& \multicolumn{1}{l}{\(\deg(R_{\pi_H})\)} & \multicolumn{1}{l}{\(\deg(R_{\pi^H})\)}\\
			\midrule
			\{\Id\} &  120g_Y + 48 n_{1} + 30 n_{2} + 45 n_{3} + 40 n_{4} \linebreak+ 50 n_{5} + 30 n_{6} - 119 & 0 & 96 n_{1} + 60 n_{2} + 90 n_{3} + 80 n_{4} \linebreak+ 100 n_{5} + 60 n_{6} \\
			\addlinespace[2ex]
			\Sym_3 & 20 g_Y + 8 n_{1} + 5 n_{2} + 15 n_{3}/2 + 6 n_{4} \linebreak+ 15 n_{5}/2 + 7 n_{6}/2 - 19 & 8 n_{4} + 10 n_{5} + 18 n_{6} & 16 n_{1} + 10  n_{2} + 15 n_{3} + 12 n_{4} \linebreak+ 15 n_{5} + 7 n_{6} \\
			\addlinespace[2ex]
			\Dih_5 & 12 g_Y + 4 n_{1} + 2 n_{2} + 4 n_{3} + 4 n_{4} + 5 n_{5} \linebreak+ 3 n_{6} - 11 & 16 n_{1} + 20 n_{2} + 10 n_{3} & 8 n_{1} + 4 n_{2} + 8 n_{3} + 8 n_{4} + 10 n_{5} + 6 n_{6} \\
			\addlinespace[2ex]
			\Alt_4 & 10 g_Y + 4 n_{1} + 2 n_{2} + 7 n_{3}/2 + 2 n_{4} \linebreak+ 7 n_{5}/2 + 5 n_{6}/2 - 9 & 12 n_{2} + 6 n_{3} + 32 n_{4} + 16 n_{5} & 8 n_{1} + 4 n_{2} + 7 n_{3} + 4 n_{4} + 7 n_{5} + 5 n_{6} \\
			\addlinespace[2ex]
			\Dih_6 & 10 g_Y + 4 n_{1} + 2 n_{2} + 7 n_{3}/2 + 3 n_{4} \linebreak+ 7 n_{5}/2 + 3 n_{6}/2 - 9 & 12 n_{2} + 6 n_{3} + 8 n_{4} + 16 n_{5} + 24  n_{6} & 8 n_{1} + 4 n_{2} + 7 n_{3} + 6 n_{4} + 7 n_{5} + 3 n_{6} \\
			\addlinespace[2ex]
			\GA & 6 g_Y + 2 n_{1} + n_{2} + 3 n_{3}/2 + 2 n_{4} \linebreak+ 5 n_{5}/2 + 3 n_{6}/2 - 5 & 16 n_{1} + 20 n_{2} + 30 n_{3} & 4 n_{1} + 2 n_{2} + 3 n_{3} + 4  n_{4} + 5 n_{5} + 3 n_{6} \\
			\addlinespace[2ex]
			\Sym_4 & 5 g_Y +  2 n_{1} + n_{2} + 3 n_{3}/2 + n_{4} \linebreak+ 3 n_{5}/2 + n_{6}/2 - 4 & 12 n_{2} + 18 n_{3} + 32 n_{4} + 28 n_{5} + 36 n_{6} & 4 n_{1} + 2 n_{2} + 3 n_{3} + 2 n_{4} + 3 n_{5} + n_{6} \\
			\addlinespace[2ex]
			\Alt_5 & 2 g_Y + n_{3}/2 +  n_{5}/2 + n_{6}/2 - 1 & 96 n_{1} + 60 n_{2} + 30 n_{3} + 80 n_{4} + 40 n_{5} & n_{3} + n_{5} + n_{6} \\
			\bottomrule
		\end{tabular}
	\end{sidewaystable}
	Also, we get that
	\begin{align}
		\deg R_{\pi^{\Sym_3}_{\Dih_6}} &= 2 n_2 + n_3 + n_5 + n_6,\label{eq:deg1}\\
		\deg R_{\pi^{\Alt_4}_{\Sym_4}} &= n_3 + n_5 + 3 n_6\label{eq:deg2}\\
		\shortintertext{and}
		\deg R_{\pi^{\Dih_5}_{\Alt_5}} &= 8 n_1 + 4 n_2 + 2 n_3 + 8 n_4 + 4 n_5.\label{eq:deg3}
	\end{align}

	According to \cite{book:FultonHarris}*{section~3.1}, there are seven complex irreducible representations of the symmetric group \(\Sym_5\):
	\begin{itemize}
		\item Two of order \(1\); the trivial representation \(U\) and the alternating one \(\tilde U\).

		\item Two of degree \(4\): the standard representation \(V\) and the product \(\tilde U \otimes V\), which we denote by \(\tilde V\).

		\item Two of degree \(5\): one that will be denoted by \(W\) and the product \(\tilde U \otimes W\) denoted by \(\tilde W\).

		\item One of degree \(6\), namely \(\Altp^2 V\).
	\end{itemize}
%
%
	The rational conjugacy classes of \(\Sym_5\) are seven: the class of \(\Id\), \((1\,2)\), \((1\,2\,3)\), \((1\,2\,3\,4)\), \((1\,2\,3\,4\,5)\), \((1\,2)(3\,4)\) and \((1\,2)(3\,4\,5)\).
	Therefore, the seven complex irreducible representations of \(\Sym_5\) are rational. \Cref{table:ComplexTableS5} is the rational character table of \(\Sym_5\).
	\begin{table}
		\caption{Complex and rational character table of \(\Sym_5\)}
		\label{table:ComplexTableS5}
		\begin{tabular}{@{} >{\(} l <{\)} *{7}{ >{\(} r <{\)} } @{}}\toprule
			&1	&10	&20	& 30 &24 &15 &20 \\
			\Alt_5	&\Id & (1\,2) &(1\,2\,3) &(1\,2\,3\,4) & (1\,2\,3\,4\,5)	&(1\,2)(3\,4)	&(1\,2)(3\,4\,5)\\
			\midrule
			U			&1	&1	&1	&1	&1 	&1	& 1 \\
			\tilde{U}	&1	&-1	&1	&-1	&1	&1	&-1 \\
			V			&4	&2	&1	&0	&-1	&0	&-1 \\
			\tilde{V}	&4	&-2	&1	&0	&-1	&0	&1  \\
			\Altp^2V	&6	&0	&0	&0	&1	&-2	&0  \\
			W			&5	&1	&-1	&-1	&0	&1	&1	\\
			\tilde W	&5	&-1	&-1	&1	&0	&1	&-1	\\
			\bottomrule
		\end{tabular}
	\end{table}
	All the rational irreducible representations of \(\Sym_5\) have Schur index \(1\) and are complex irreducible representations.
	Hence, by \cref{theo:GroupAlgebraDecomposition}, the group algebra decomposition of \(\J(\hat X)\) is of the form
	\begin{equation}\label{eq:S5DecompositionForm}
		\J(\hat X) \sim A_1 \times A_2 \times A_3^4 \times A_4^4 \times A_5^6 \times A_6^5 \times A_7^5.
	\end{equation}

	\begin{theo}\label{theo:CaseS5}
		If \(G = \Sym_5\), then
		\begin{equation}\label{eq:S5Decomposition}
			\begin{split}
			\J(\hat X) \sim{} &\J(Y) \times \Prym(\pi^{\Alt_5})\times \Prym(f)^4 \times \Prym\paren{\pi^{\Alt_4}_{\Sym_4}, \pi^{\Alt_4}_{\Alt_5}}^4\\
			&\times \Prym\paren{\pi^{\Sym_3}_{\Dih_6}, \pi^{\Sym_3}_{\Sym_4}}^6 \times \Prym\paren{\pi^{\Dih_5}_{\Alt_5}, \pi^{\Dih_5}_{\GA}}^5 \times \Prym(\pi^{\GA})^5
		\end{split}
		\end{equation}
		is the group algebra decomposition of \(\J(\hat X)\), where \(\Sym_5\) acts as multiples of \(U\), \(\tilde U\), \(V\), \(\tilde V\), \(\Altp^2 V\), \(W\) and \(\tilde W\) on the respective components of the decomposition. The dimensions of the Prym varieties are:
		\begin{itemize}
			\item \(\dim \Prym(\pi^{\Alt_5}) = g + n_3/2 + n_5/2 + n_6/2 - 1\)

			\item \(\dim \Prym(f) = 4g + 2n_1 + n_2 + 3n_3/2 + n_4 + 3n_5/2 + n_6/2 - 4\)

			\item \(\dim \Prym\paren{\pi^{\Alt_4}_{\Sym_4}, \pi^{\Alt_4}_{\Alt_5}}= 4g + 2n_1 + n_2 + 3n_3/2 + n_4 + 3n_5/2 + 3n_6/2 - 4\)

			\item \(\dim \Prym\paren{\pi^{\Sym_3}_{\Dih_6}, \pi^{\Sym_3}_{\Sym_4}} = 6g + 2n_1 + 2n_2 + 5n_3/2 + 2n_4/2 + 5n_5/2 + 3n_6/2 - 6\)

			\item \(\dim \Prym\paren{\pi^{\Dih_5}_{\Alt_5}, \pi^{\Dih_5}_{\GA}} = 5g + 2n_1 + n_2 + 2n_3 + 2n_4 + 2n_5 + n_6 - 5\)

			\item \(\dim \Prym(\pi^{\GA}) = 5g + 2n_1 + n_2 + 3n_3/2 + 2n_4 + 5n_5/2 + 3n_6/2 - 5\)
		\end{itemize}
		and the types of their polarizations are:
		\begin{itemize}
			\item \(\type \Theta_{\Prym(f)} = (1_{\times (3g + 2n_1 + n_2+ 3n_3/2 + n_4 + 3n_5/2 + n_6/2 - 4)}, 5_{\times g})\);

			\vspace{1ex}

			\item \(\type \Theta_{\Prym(\pi^{\GA})} = (1_{\times (4g + 2n_1+n_2+ 3n_3/2 + 2n_4 + 5n_2/2 + 3n_6/2 - 5)}, 6_{\times g})\);

			\vspace{1ex}

			\item \(\type \Theta_{\Prym(\pi^{\Alt_5})} = \begin{cases}
				(2_{\times g - 1})	&\text{if \(n_3 = n_5 = n_6 = 0\),}\\
				(1_{\times ((n_3 + n_5 + n_6)/2 - 1)}, 2_{\times g})	&\text{if \(n_3 + n_5 + n_6 > 0\).}
			\end{cases}\)
		\end{itemize}
	\end{theo}
	\begin{proof}
		Frobenius reciprocity and \cref{table:ComplexTableS5} show that \(\rho_{\GA} = U \oplus \tilde W\), \(\rho_{\Alt_5} = U \oplus \tilde U\), \(\rho_{\Sym_4} = U \oplus V\), \(\rho_{\Dih_6} = U \oplus V \oplus W\), \(\rho_{\Dih_5} = U \oplus \tilde U \oplus W \oplus \tilde W\), \(\rho_{\Alt_4} = U\oplus \tilde U \oplus V \oplus \tilde V\) and \(\rho_{\Sym_3} = U \oplus 2V \oplus \Altp^2 V \oplus W\).
		Therefore, following the notation of \cref{eq:S5DecompositionForm}, \cref{prop:GAComponentAsPrym} implies that \(A_1 \sim \J(Y)\), \(A_2 \sim \Prym(\pi^{\Alt_5})\), \(A_3 \sim \Prym(f)\) and \(A_7 \sim \Prym\paren{\pi^{\GA}}\).
		Also, using \cref{theo:PrymOfPairsOfCoverings}, we get \(A_4 \sim \Prym\paren{\pi^{\Alt_4}_{\Sym_4}, \pi^{\Alt_4}_{\Alt_5}}\), \(A_5 \sim \Prym\paren{\pi^{\Sym_3}_{\Dih_6}, \pi^{\Sym_3}_{\Sym_4}}\) and \(A_6 \sim \Prym\paren{\pi^{\Dih_5}_{\Alt_5}, \pi^{\Dih_5}_{\GA}}\). The dimension of the several Jacobian and Prym varieties follows from \cref{table:TotalRamificationsS5,eq:deg1,eq:deg2,eq:deg3,prop:DimForPrymOfPairs}.

		According to \cref{table:TotalRamificationsS5}, we have that \(\pi^{\Alt_5}\) is étale if and only if \(n_3 = n_5 = n_6 = 0\); besides, the map \(\pi^{\Alt_5}\) is cyclic (because it is of degree~\(2\)) whereas \(f\) and \(\pi^{\GA}\) are not (neither \(\Sym_4\) nor \(\GA\) are normal in \(\Sym_5\)); also, since neither \(\Sym_4\) nor \(\GA\) are subgroups of \(\Alt_5\), which is the only normal subgroup of \(\Sym_5\), the maps \(f\) and \(\pi^{\GA}\) do not factor by a Galois covering onto \(Y\). Hence, by \cref{prop:CyclicEtaleFactor}, we have that \(\ord{\ker f^*} = 1\), \(\ord{\ker \pi^{\GA*}} = 1\) and
		\begin{equation*}
			\ord{\ker \pi^{\Alt_5*}} = \begin{cases}
				2& \text{if \(n_3 = n_5 = n_6 = 0\),}\\
				1& \text{if \(n_3 + n_5 + n_6 > 0\);}
			\end{cases}
		\end{equation*}
		From \cref{item:DecompositionInJacAndPrym} of \cref{theo:PrymAndPolarizations} and \cref{table:TotalRamificationsS5}, we get the desired polarization types.
	\end{proof}

	\bibliography{References}

\begin{bibdiv}
\begin{biblist}

\bib{book:Lange}{book}{
      author={Birkenhake, Christina},
      author={Lange, Herbert},
       title={Complex abelian varieties},
     edition={Second},
      series={Fundamental Principles of Mathematical Sciences},
   publisher={Springer-Verlag, Berlin},
        date={2004},
      volume={302},
        ISBN={3-540-20488-1},
      review={\MR{2062673}},
}

\bib{book:Breuer}{book}{
      author={Breuer, Thomas},
       title={Characters and automorphism groups of compact {R}iemann
  surfaces},
      series={London Mathematical Society Lecture Note Series},
   publisher={Cambridge University Press, Cambridge},
        date={2000},
      volume={280},
        ISBN={0-521-79809-4},
      review={\MR{1796706}},
}

\bib{paper:Broughton}{article}{
      author={Broughton, S.~Allen},
       title={Classifying finite group actions on surfaces of low genus},
        date={1991},
        ISSN={0022-4049},
     journal={J. Pure Appl. Algebra},
      volume={69},
      number={3},
       pages={233\ndash 270},
         url={https://ezproxy.ufro.cl:2069/10.1016/0022-4049(91)90021-S},
      review={\MR{1090743}},
}

\bib{paper:TransitiveGroups}{article}{
      author={Butler, Gregory},
      author={McKay, John},
       title={The transitive groups of degree up to eleven},
        date={1983},
        ISSN={0092-7872},
     journal={Comm. Algebra},
      volume={11},
      number={8},
       pages={863\ndash 911},
      review={\MR{695893}},
}

\bib{paper:CR2006}{article}{
      author={Carocca, Angel},
      author={Rodr\'{\i}guez, Rub\'{\i}~E.},
       title={Jacobians with group actions and rational idempotents},
        date={2006},
        ISSN={0021-8693},
     journal={J. Algebra},
      volume={306},
      number={2},
       pages={322\ndash 343},
         url={https://ezproxy.ufro.cl:2069/10.1016/j.algebra.2006.07.027},
      review={\MR{2271338}},
}

\bib{book:FultonHarris}{book}{
      author={Fulton, William},
      author={Harris, Joe},
       title={Representation theory},
      series={Graduate Texts in Mathematics},
   publisher={Springer-Verlag, New York},
        date={1991},
      volume={129},
        ISBN={0-387-97527-6; 0-387-97495-4},
         url={https://ezproxy.ufro.cl:2069/10.1007/978-1-4612-0979-9},
        note={A first course, Readings in Mathematics},
      review={\MR{1153249}},
}

\bib{book:Girondo}{book}{
      author={Girondo, Ernesto},
      author={Gonz\'{a}lez-Diez, Gabino},
       title={Introduction to compact {R}iemann surfaces and dessins
  d'enfants},
      series={London Mathematical Society Student Texts},
   publisher={Cambridge University Press, Cambridge},
        date={2012},
      volume={79},
        ISBN={978-0-521-74022-7},
      review={\MR{2895884}},
}

\bib{software:GAP}{manual}{
      author={Group, The~GAP},
       title={{GAP---Groups, Algorithms, and Programming, Version 4.11.0}},
        date={2020},
        note={\url{https://www.gap-system.org}},
}

\bib{paper:LangeRecillas04(2)}{article}{
      author={Lange, Herbert},
      author={Recillas, Sevín},
       title={Abelian varieties with group action},
        date={2004},
        ISSN={0075-4102},
     journal={J. Reine Angew. Math.},
      volume={575},
       pages={135\ndash 155},
         url={https://ezproxy.ufro.cl:2069/10.1515/crll.2004.076},
      review={\MR{2097550}},
}

\bib{paper:LangeRecillas04}{article}{
      author={Lange, Herbert},
      author={Recillas, Sevín},
       title={Prym varieties of pairs of coverings},
        date={2004},
        ISSN={1615-715X},
     journal={Adv. Geom.},
      volume={4},
      number={3},
       pages={373\ndash 387},
         url={https://ezproxy.ufro.cl:2069/10.1515/advg.2004.022},
      review={\MR{2071812}},
}

\bib{paper:recillasTri}{article}{
      author={Recillas, Sevín},
       title={Jacobians of curves with \(g_4^1\)'s are the prym's of trigonal
  curves},
        date={1974},
     journal={Bol. Soc. Mat. Mexicana},
}

\bib{paper:RR98}{article}{
      author={Recillas, Sevín},
      author={Rodríguez, Rubí~E.},
       title={Jacobians and representations of \(s_3\)},
        date={1998},
     journal={Aportaciones Mat. Inv.},
      volume={13},
       pages={117\ndash 140},
      eprint={math/0303155},
}

\bib{paper:RR03}{article}{
      author={Recillas, Sevín},
      author={Rodríguez, Rubí~E.},
       title={Prym varieties and fourfold covers},
        date={2003},
      eprint={math/0303155},
}

\bib{thesis:AnitaRojas}{thesis}{
      author={Rojas, Anita},
       title={Acciones de grupos en variedades {J}acobianas},
        type={phd},
 institution={Pontificia Universidad Católica de Chile},
        date={2002},
  note={\url{http://www.mat.uc.cl/doctorado-en-matematica-graduados-y-tesis.html}},
}

\bib{paper:Sanchez}{incollection}{
      author={S\'{a}nchez-Arg\'{a}ez, Armando},
       title={Actions of the group {$A_5$} in {J}acobian varieties},
        date={1999},
   booktitle={X{XXI} {N}ational {C}ongress of the {M}exican {M}athematical
  {S}ociety ({S}panish) ({H}ermosillo, 1998)},
      series={Aportaciones Mat. Comun.},
      volume={25},
   publisher={Soc. Mat. Mexicana, M\'{e}xico},
       pages={99\ndash 108},
      review={\MR{1790528}},
}

\bib{book:Serre}{book}{
      author={Serre, Jean-Pierre},
       title={Linear representations of finite groups},
   publisher={Springer-Verlag, New York-Heidelberg},
        date={1977},
        ISBN={0-387-90190-6},
        note={Translated from the second French edition by Leonard L. Scott,
  Graduate Texts in Mathematics, Vol. 42},
      review={\MR{0450380}},
}

\bib{software:SageMath}{manual}{
      author={{The Sage Developers}},
       title={{S}agemath, the {S}age {M}athematics {S}oftware {S}ystem,
  {V}ersion 9.3},
        date={2021},
        note={\url{https://www.sagemath.org}},
}

\end{biblist}
\end{bibdiv}
\end{document}